\documentclass{amsart}
 
\usepackage{amsmath}
\usepackage{amssymb}
\usepackage{amsfonts}

\usepackage{alltt}
\usepackage{amsthm}
\usepackage{bm}
\usepackage{mathrsfs}
\usepackage{graphicx}
\usepackage{caption}
\usepackage{color}
\usepackage{mathtools}
\usepackage[T1]{fontenc}
\usepackage[utf8]{inputenc}
\usepackage{cases}
\usepackage{xcolor}
\usepackage[colorlinks,linkcolor=blue,urlcolor=blue,citecolor=blue,anchorcolor=blue]{hyperref}
\usepackage{longtable}

\usepackage{graphicx} 
	\usepackage{tikz,pgfplots}
	\usetikzlibrary{calc,patterns,decorations.pathmorphing,decorations.markings}
	\usepackage{pgfplotstable}
	\pgfplotsset{compat=newest}

\usepackage{matlab-prettifier}
\newcommand{\R}{\mathbb R}

\newcommand{\im}{\operatorname{Im}}

\newcommand{\sinc}{\operatorname{sinc}}

\newcommand{\PP}{\mathsf{P}}
\newcommand{\DDj}{\mathsf{D}^j}

\newcommand{\DDt}{\widetilde{\mathsf{D}}}
\newcommand{\DDe}{\mathsf{D}}
\newcommand{\FT}{\mathcal{T}}
\newcommand{\IFT}{\mathcal{I}}
\newcommand{\ITDT}{\operatorname{\mathcal{I}_\mathrm{disc}}}
\newcommand{\FTDT}{\operatorname{\mathcal{T}_\mathrm{disc}}}

\newcommand{\q}{q_{0}}
\newcommand{\On}{\Omega_{0}}
\newcommand{\M}{\mathsf{M}}
\newcommand{\infint}{\int_{-\infty}^\infty}

\numberwithin{equation}{section}

\newtheorem{theorem}{Theorem}[section]
\newtheorem{proposition}[theorem]{Proposition}
\newtheorem{lemma}[theorem]{Lemma}

\newtheorem*{thm*}{Theorem}
\newtheorem*{lemma*}{Lemma}
\newtheorem*{prop*}{Proposition}

\theoremstyle{definition}
\newtheorem{dfn}[theorem]{Definition}
\newtheorem{example}[theorem]{Example}

\newtheorem*{dfn*}{Definition}

\theoremstyle{remark}
\newtheorem{remark}[theorem]{Remark}

\newtheorem*{rmk}{Remark}

\newtheorem*{rmk1}{Fourier integral operators}
\newtheorem*{rmk2}{Initial conditions}
\newtheorem*{rmk3}{Backward and forward type schemes}
\newtheorem*{rmk4}{Nonmatching finite difference schemes}

 \usepackage{pgfplotstable}

\makeatletter
\let\save@mathaccent\mathaccent
\newcommand*\if@single[3]{%
  \setbox0\hbox{${\mathaccent"0362{#1}}^H$}%
  \setbox2\hbox{${\mathaccent"0362{\kern0pt#1}}^H$}%
  \ifdim\ht0=\ht2 #3\else #2\fi
  }
\newcommand*\rel@kern[1]{\kern#1\dimexpr\macc@kerna}
\newcommand*\widebar[1]{\@ifnextchar^{{\wide@bar{#1}{0}}}{\wide@bar{#1}{1}}}
\newcommand*\wide@bar[2]{\if@single{#1}{\wide@bar@{#1}{#2}{1}}{\wide@bar@{#1}{#2}{2}}}
\newcommand*\wide@bar@[3]{%
  \begingroup
  \def\mathaccent##1##2{%
    \let\mathaccent\save@mathaccent
    \if#32 \let\macc@nucleus\first@char \fi
    \setbox\z@\hbox{$\macc@style{\macc@nucleus}_{}$}%
    \setbox\tw@\hbox{$\macc@style{\macc@nucleus}{}_{}$}%
    \dimen@\wd\tw@
    \advance\dimen@-\wd\z@
    \divide\dimen@ 3
    \@tempdima\wd\tw@
    \advance\@tempdima-\scriptspace
    \divide\@tempdima 10
    \advance\dimen@-\@tempdima
    \ifdim\dimen@>\z@ \dimen@0pt\fi
    \rel@kern{0.6}\kern-\dimen@
    \if#31
      \overline{\rel@kern{-0.6}\kern\dimen@\macc@nucleus\rel@kern{0.4}\kern\dimen@}%
      \advance\dimen@0.4\dimexpr\macc@kerna
      \let\final@kern#2%
      \ifdim\dimen@<\z@ \let\final@kern1\fi
      \if\final@kern1 \kern-\dimen@\fi
    \else
      \overline{\rel@kern{-0.6}\kern\dimen@#1}%
    \fi
  }%
  \macc@depth\@ne
  \let\math@bgroup\@empty \let\math@egroup\macc@set@skewchar
  \mathsurround\z@ \frozen@everymath{\mathgroup\macc@group\relax}%
  \macc@set@skewchar\relax
  \let\mathaccentV\macc@nested@a
  \if#31
    \macc@nested@a\relax111{#1}%
  \else
    \def\gobble@till@marker##1\endmarker{}%
    \futurelet\first@char\gobble@till@marker#1\endmarker
    \ifcat\noexpand\first@char A\else
      \def\first@char{}%
    \fi
    \macc@nested@a\relax111{\first@char}%
  \fi
  \endgroup
}
\makeatother

\makeatletter
\def\@setauthors{%
  \begingroup
  \def\thanks{\protect\thanks@warning}%
  \trivlist
  \centering\large \@topsep30\p@\relax
  \advance\@topsep by -\baselineskip
  \item\relax
  \author@andify\authors
  \def\\{\protect\linebreak}%
  \authors%
  \ifx\@empty\contribs
  \else
    ,\penalty-3 \space \@setcontribs
    \@closetoccontribs
  \fi
  \endtrivlist
  \endgroup
}
\def\@settitle{\begin{center}%
  \baselineskip14\p@\relax
    \normalfont\LARGE
  \@title
  \end{center}%
}
\makeatother

\pgfplotsset{compat=newest}

\begin{document}

\date{\today}

\title[Removing numerical dispersion from linear evolution equations]{Removing numerical dispersion from linear evolution equations}

\author[J. Wittsten]{Jens Wittsten}
\address[Jens Wittsten]{Centre for Mathematical Sciences, Lund University, Sweden and Department of Engineering, University of Bor{\aa}s, Sweden}
\email{jensw@maths.lth.se}

\author[E. Koene]{Erik F. M. Koene}
\address[Erik F. M. Koene]{Institute of Geophysics, ETH-Z{\"u}rich, Switzerland}
\email{erik.koene@erdw.ethz.ch}

\author[F. Andersson]{Fredrik Andersson}
\address[Fredrik Andersson]{Institute of Geophysics, ETH-Z{\"u}rich, Switzerland}
\email{fredrik.andersson@erdw.ethz.ch}

\author[J. Robertsson]{Johan O. A. Robertsson}
\address[Johan O. A. Robertsson]{Institute of Geophysics, ETH-Z{\"u}rich, Switzerland}
\email{johan.robertsson@erdw.ethz.ch}

\begin{abstract}
We describe a method for removing the numerical errors in the modeling of linear evolution equations that are caused by approximating the time derivative by a finite difference operator. The method is based on integral transforms realized as certain Fourier integral operators, called time dispersion transforms, and we prove that, under an assumption about the frequency content, it yields a solution with correct evolution throughout the entire lifespan. We demonstrate the method on a model equation as well as on the simulation of elastic and viscoelastic wave propagation.
\end{abstract}



\maketitle

\section{Introduction}

The difference between a continuous differential equation and its discretized counterpart is a source of numerical artifacts. Generally, the discretized system differs from the intended system in its dispersive and dissipative properties, so errors in the computation are referred to as \textit{numerical dispersion} and \textit{numerical dissipation} \cite{anderson1984}.
Here dispersion refers to a process in which energy separates into its component frequencies as the solution evolves, while dissipation refers to damping of energy during the evolution. Numerical dispersion thus refers to phase errors, while numerical dissipation refers to amplitude errors. The combined effect of the two numerical errors is sometimes described as \textit{numerical diffusion}.

Numerical diffusion errors are typically studied through the local truncation error, i.e., the consistency between the discrete and continuous equation in terms of the discrete step size. If the method is stable, the Lax equivalence theorem \cite{lax1956survey} implies that the discretized equation converges to its continuous counterpart. As a consequence, the majority of the numerical methods for differential equations are designed with the intent of minimizing the local truncation error, with the expectation that the global error will then also be small. Examples are high-order accurate derivative schemes \cite{holberg1987computational}, \cite{fornberg1998practical} and high-order accurate integration schemes such as Runge-Kutta or ADER (Arbitrary high order schemes using DERivatives) \cite{hoffmann2000computational}, \cite{titarev2002ader}, \cite{hairer2006geometric}, \cite{leveque2007finite}. The high-order techniques typically lead to more accurate results compared to low-order methods, but come with a trade-off in increased computational cost.

In this paper we analyze numerical dispersion errors not through the local error but by comparing numerical solutions to true solutions. 
To illustrate, consider the evolution equation
\begin{align}\label{eq:introevo}
u'(t)+Au(t)&=f(t),\quad t>0,\\
u(t)&\equiv0,\quad t\le0,
\label{eq:introinitialcondition}
\end{align}
where $A$ is a linear operator independent of time $t$. Arguing formally we find by taking Fourier transforms with respect to $t$ of both sides of \eqref{eq:introevo} that
\begin{equation}\label{eq:introFourier1}
2\pi i \omega\widehat u(\omega)+A\widehat u(\omega)=\widehat f(\omega),
\end{equation}
where $\widehat f(\omega)=\int_{-\infty}^\infty e^{-2\pi i t\omega}f(t)\,dt$. Suppose now that we want to simulate  the solution $u$ by obtaining a solution $v$ of the finite difference equation 
\begin{align}\label{eq:introFD}
\frac{v(t+\Delta t)-v(t-\Delta t)}{2\Delta t}+Av(t)&=f(t),\quad t>0,\\
v(t)&\equiv0,\quad t\le0.
\end{align}
The finite difference approximation of the time derivative introduces an error which can be expressed by taking Fourier transforms and noting that $\widehat v$ will not satisfy \eqref{eq:introFourier1} but instead
\begin{equation}\label{eq:introFourier2}
2\pi i q(\omega)\widehat v(\omega)+A\widehat v(\omega)=\widehat f(\omega),
\end{equation}
where 
$$
q(\omega)=\frac{1}{2\pi i}\frac{e^{2\pi i \omega\Delta t}-e^{-2\pi i\omega\Delta t}}{2\Delta t}=\frac{\sin(2\pi\omega\Delta t)}{2\pi\Delta t}.
$$
Such functions appear as (part of) the so-called amplification factor in von Neumann stability analysis of spatial finite difference operators \cite{leveque2007finite}, but here it is rotated from the spatial side into the time direction where it will be called a {\it phase shift function}. The connection is explored in Remark \ref{rmk:vonNeumann} below. As the name signifies, comparison between $q(\omega)$ and $\omega$ allows for a description of the numerical dispersion (i.e., phase) errors resulting from the finite difference approximation of the time derivative.
We now make the simple observation that if \eqref{eq:introFourier1} is evaluated at $q(\omega)$ instead of at $\omega$ then $\widehat u$ is seen to satisfy
\begin{equation*}
2\pi i q(\omega)\widehat u(q(\omega))+A\widehat u(q(\omega))=\widehat f(q(\omega)).
\end{equation*}
By comparing this equation with \eqref{eq:introFourier2} we find that if $v$ is to correctly simulate the evolution of $u$ then the governing equation \eqref{eq:introFD} for $v$ should be modified by replacing $f(t)$ with the function $g(t)$ that has Fourier transform $\widehat g(\omega)=\widehat f(q(\omega))$, while $u$ and $v$ should satisfy the relation $\widehat v(\omega)=\widehat u(q(\omega))$, if possible. Since $q$ is periodic, one obstruction is that unless \eqref{eq:introevo}--\eqref{eq:introinitialcondition} leads to a solution $u$ which is bandlimited in time,
the relation $\widehat v(\omega)=\widehat u(q(\omega))$ will only capture part of the frequency content of the sought solution $u$. (Here $u$ is said to be bandlimited in time when $\widehat u$ has compact support.) As we shall see, this obstruction can be made negligible also for non-bandlimited $u$ as long as $\widehat u$ and $\widehat f$ has sufficiently rapid decay at infinity. The transform $f\mapsto g$ which modifies the finite difference equation is called the {\it forward time dispersion transform} (FTDT) and the transform $v\mapsto u$ which (ideally) turns a solution of the altered finite difference equation into a solution of the original evolution equation is called the {\it inverse time dispersion transform} (ITDT).   
In practice, there will typically be stability criteria placing an upper bound on the step size $\Delta t$. Throughout the paper we will assume that $\Delta t$ is chosen with such constraints in mind, but as we will demonstrate, one benefit of the proposed method is that it allows $\Delta t$ to be chosen at such upper bounds while still yielding very accurate results. This is true even for low order finite difference schemes, thus providing a new and computationally cheap way to obtain good simulation results.

The method of using dispersion transforms to correct for numerical errors caused by finite difference approximations of time derivatives was introduced in geophysics for the wave equation, see Stork \cite{stork2013eliminating} for the original approach. It was then further developed and improved by many authors 
\cite{wang2015finite}, 
\cite{gao2016third}, 
\cite{koene2017removing}, 
\cite{koene2018eliminating}, 
\cite{amundsen2019elimination}, 
\cite{koene2019eliminating}, 
\cite{mittet2019second} 
and seems to now have settled on the definitions of the transforms presented here. For a review of the development we refer to Koene {\it et al.}~\cite{koene2018eliminating}. We mention that in the geophysical literature, usage of the FTDT and ITDT is usually described as applying a pre-computation and a post-computation filter, respectively. Section \ref{section:theory} presents a rigorous mathematical treatment of the method and shows using a unified approach that it has applications to a large class of linear evolution equations. After formally introducing the dispersion transforms (Definition \ref{def:transforms}) we use them to establish the correspondence between an evolution equation and its finite difference equivalent hinted at in the example above (Theorem \ref{thm:evo}), and show that they can be interpreted as Fourier integral operators. In \S\ref{ss:discreteITDT} we provide discrete versions of the transforms and discuss implementation. Our main result shows that the proposed method of using the discrete dispersion transforms as pre- and post-computational filters yields a numerical solution that correctly models the desired evolution for any length of time (see Theorem \ref{thm:mainthm} for the precise statement). Section \ref{section:numsim} demonstrates the theoretic results by conducting numerical tests on a model equation where the solution obtained by the proposed method compares to the analytic solution with double precision accuracy (see Figure \ref{fig1}). 
The approximation error is small as long as the frequency content of the source term is negligible outside a window defined by the range of the phase function $q$ (see Figure \ref{fig2}). The size of this window becomes arbitrarily large as $\Delta t\to0$.

One shortcoming of the method is the need to store the solution for all moments in time, which for two- or three-dimensional problems may require a very large memory capacity. However, even though simulations might be carried out over a global domain the solution is often only desired at a small subset to which the corrections by means of the dispersion transforms can be restricted, at a cost that is far smaller than running the simulation with a globally large accuracy. We demonstrate such an application in Section \ref{section:viscoelastic}, where we perform elastic and viscoelastic wave simulation for the earth's subsurface, but only use the dispersion transforms to correct the resulting ground motion at certain points on the boundary (the source and receiver positions). The viscoelastic simulations show that the method can deal even with dissipative wave physics while still yielding highly accurate solutions.

The paper is concluded with three appendices. In Appendix A we have gathered results of tangential or supplementary nature referenced in the main text. In Appendix B one can find the implementation of the finite difference scheme used in the viscoelastic wave simulations. Finally, in Appendix C we provide codes for implementing the dispersion transforms in MATLAB.

\section{Numerical dispersion in evolution equations}\label{section:theory}
Let $ X\subset\R^d$ and $u=(u_1,\ldots,u_K)$ be a vector valued function of $(t,x)\in \R\times X$. Introduce the $K\times K$ system of differential operators
$$
\mathscr P_iu(t,x)=\partial_t^{n_i}u_i(t,x)+\sum_{j=0}^{n_i-1}\sum_{k=1}^K 
\partial_t^j
L_{ijk} u_k(t,x),\quad 1\le i\le K,
$$
where $\partial_t^ju_k(t,x)=\partial^ju_k(t,x)/\partial t^j$ and the $L_{ijk}$ are linear spatial operators depending on $x\in X$ but independent of time $t$ so that $\partial_t$ and $L_{ijk}$ commute. Consider the evolution equation
in $\R\times X$ given by
\begin{align}\label{evo}
\mathscr P_iu(t,x)&=f_i(t,x),\quad 1\le i\le K,\\
 u(t,x) &\equiv 0,\quad t\le t_0.
\label{init}
\end{align}
Since the system is translation invariant in $t$ we may without loss of generality assume that $t_0=0$ below. 
We will assume that the problem is well posed and that, depending on the spatial operators $L_{ijk}$, appropriate spatial conditions are imposed to ensure a unique solution. For an extensive background on partial differential equations we refer to H{\"o}rmander \cite{hormander1983linear} and Evans \cite{evans2010partial}.

When solving \eqref{evo}--\eqref{init} by means of finite difference methods, numerical dispersion errors inevitably occur as a result of approximating the time derivatives with finite differences. The purpose of this paper is to establish a method by which to alter the chosen finite difference system and capture the correct time evolution of the solution $u$ to \eqref{evo}--\eqref{init}.

In this work, the exact structure of the spatial operators will not be essential. In the applications we have in mind, each $L_{ijk}$ will typically be a differential operator in $x$ with coefficients that depend on $x$ but not on $t$, or (going one step further in the direction of obtaining a numerical solution) each $L_{ijk}$ will be the result of discretizing such a differential operator by means of some numerical scheme. To isolate the effects of time dispersion we will in the latter case assume that any resulting space dispersion errors are essentially fully decoupled from the time dispersion errors,
and that appropriate stability criteria 
that may govern the possible size of the time step $\Delta t$ are satisfied. We will only be concerned with finite difference schemes which are numerically stable and depend continuously on the initial data. A comprehensive treatment of finite difference methods can be found in LeVeque \cite{leveque2007finite}. 

\begin{example}\label{ex:waveq}
As a prototype, consider the scalar, one-dimensional acoustic wave equation in heterogeneous media:
$$
\partial_t^2 u(t,x)-c(x)\partial_x^2u(t,x)=f(t,x),\quad t>0,\quad 0\le x\le 1,
$$
where $u$ is the acoustic pressure and the propagation velocity $c(x)$ depends on position $x$. Assume Dirichlet boundary values $u(t,0)=u(t,1)=0$. This is an example of \eqref{evo} with $K=1$ and $\mathscr P u$ of the form $\mathscr Pu=\partial_t^2u+L_0u+\partial_t L_1u$, where $L_0=-c(x)\partial_x^2$ and $L_1\equiv0$. On the other hand, if we want to solve this equation numerically, we may choose to discretize the equation in $x$ by sampling at $x=x_j$ where $x_j=j\Delta x$ for $j=0,1,\ldots,N+1$, and replace $\partial_x^2$ with a finite difference, say
$$
\partial_t^2u(t,x_j)-c(x_j)\frac{u(t,x_{j-1})-2u(t,x_j)+u(t,x_{j+1})}{\Delta x^2}=f(t,x_j)
$$
for $1\le j\le N$, with the boundary values dictating $u(t,x_0)=u(t,x_{N+1})=0$. 
This is also an example of \eqref{evo} with $K=1$ and $\mathscr P u$ of the form $\mathscr Pu=\partial_t^2u+L_0u+\partial_t L_1u$, where the unknown $u$ is now the $t$-dependent vector $u(t)=(u(t,x_1),\ldots,u(t,x_{N}))$ and while we still have $L_1\equiv0$, $L_0$ is now a tridiagonal matrix which factorizes as
$$
L_0=-\frac{1}{\Delta x^2}\begin{pmatrix*}[r] c(x_1)&&& \\ &c(x_2)&&\\&&\ddots&\\&&&c(x_{N})\end{pmatrix*}\begin{pmatrix*}[r] -2&1& &\\ 1&-2&1&\\&\ddots&\ddots&\ddots\\&&1&-2\end{pmatrix*}.
$$
Whenever we discuss solving this type of equation by means of a finite difference scheme in time of step size $\Delta t$, we will always assume that the corresponding Courant-Friedrich-Lewy (CFL) condition is satisfied by $\Delta t$ in terms of the interval length $\Delta x$, so that the resulting equation is numerically stable. 
\end{example}

Note that because of the initial condition \eqref{init}, the source terms $f_i$ in \eqref{evo} will have to vanish identically for $t\le0$. We will in addition assume they have sufficient decay as $t\to\infty$ and that each $L_{ijk}$ and $f_i$ are regular enough  for \eqref{evo}--\eqref{init} to admit a strong solution $u_i$, integrable with respect to $t$, such that for each $1\le i\le K$
\begin{equation}\label{regularityassumption}
t\mapsto \sup_{x\in X}\lvert u_i(t,x)\rvert\in H^{n_i}(\R).
\end{equation}
Here, $H^s=W^{s,2}$ is the usual $L^2$ Sobolev space of order $s$. 
In particular, the partial Fourier transforms $\widehat{u}_i(\omega,x)$ and $\widehat{f}_i(\omega,x)$ are assumed to be well defined and square integrable, where
$$
\widehat{u}_i(\omega,x)=\mathcal F(u_i(\cdot,x))(\omega)=\infint e^{-2\pi i t\omega}u_i(t,x)\, dt.
$$
We remark that \eqref{regularityassumption} does not automatically hold in general -- in the example \eqref{eq:introevo}--\eqref{eq:introinitialcondition} discussed in the introduction it would depend on properties of the operator $A$. In the applications we focus on in this paper, however, the energy is expected to dissipate in the region of interest, due e.g.~to damping or the geometry of the spatial domain, which makes \eqref{regularityassumption} natural.

\begin{rmk}
Realizations of the Cauchy problem \eqref{evo}--\eqref{init} can for example be found in initial value problems for:
\begin{itemize}
\item[(1)] Ordinary differential equations with constant coefficients.
\item[(2)] Heat equations, linear parabolic equations.
\item[(3)] Wave equations, linearly damped wave equations, Maxwell's equations, linear elasticity.
\item[(4)] Visco-acoustic and viscoelastic equations solved via memory variables (see \S\ref{ss:memory}).
\item[(5)] Strictly hyperbolic pseudodifferential equations, Tricomi equations.
\end{itemize}
\end{rmk}

\subsection{Finite difference system}\label{ss:FD}

Let $\DDe$ denote the finite difference operator
\begin{equation}\label{1order}
\DDe  v_i(t,x)= \sum_{n\in Z} c_{1,n} v_i(t+n\Delta t,x),
\end{equation}
where $Z$ is a finite set usually consisting of a subset of the integers or half-integers, and the coefficients  $c_{1,n}$ are chosen so that $\DDe$ becomes an approximation of the first order time derivative $\partial_t$. Introduce the finite difference operators
\begin{equation}\label{eq:defofFDsystem}
\PP_iv(t,x)=
\mathsf{D}^{n_i}v_i(t,x)+\sum_{j=0}^{n_i-1}\sum_{k=1}^K \DDj L_{ijk} v_k(t,x),\quad 1\le i\le K,
\end{equation}
corresponding to the differential operators $\mathscr P_i$ discussed above, obtained by approximating the time derivatives by means of $\DDe$. Note that the spatial operators $L_{ijk}$ thus are the same in $\mathscr P_i$ and $\PP_i$.
Here and for the majority of the paper $\DDj$ denotes the composition
$$
\DDj=\underbrace{\DDe\circ\ldots\circ\DDe}_j,
$$
which means in particular that the same scheme $\DDe$ is assumed to be used as a basis for $\DDj$ in each of the $K$ operators $\mathsf P_i$. The case of non-matching finite difference schemes is discussed briefly on page \pageref{rmk:nonmatching} below and again in \S\ref{ss:nonmatching} in the appendix.

Taking a partial Fourier transform of \eqref{1order} we observe that
\begin{equation}\label{FourierDj}
\infint e^{-2\pi i t\omega}\DDj v_i(t,x) \, dt = \bigg( \sum_n c_{1,n}  e^{2\pi i \omega \Delta t n}\bigg)^j\widehat{v}_i(\omega,x).
\end{equation}
In view of this identity and the fact that
\begin{equation*}
\mathcal F\left(\partial_t^j u_i(\cdot,x)\right)(\omega)
=(2\pi i\omega)^j\widehat u_i(\omega,x),
\end{equation*}
we define a {\it phase shift function} $q$ as
\begin{equation}\label{def:q}
q(\omega) = \frac{1}{2\pi i}\sum_n c_{1,n}  e^{2\pi i n\omega \Delta t } 
\end{equation}
so that
\begin{equation*}
\mathcal F\left(\DDj v_i(\cdot,x)\right)(\omega)
=(2\pi iq(\omega))^j\widehat v_i(\omega,x).
\end{equation*}
We will assume that $c_{1,n}$ is chosen in such a way that $q(\omega)$ is real-valued and invertible as a mapping $q:\Omega\to q(\Omega)$ 
for some subset $\Omega=\Omega(\Delta t)\subset\R$.
For a comment on the case when $q$ is not real-valued (which happens e.g., in the case of a forward Euler scheme), see the remark on page \pageref{ss:nonreal}. Note also that with respect to the normalized variable $\omega\Delta t$, the right-hand side of \eqref{def:q} is invertible for all $\omega\Delta t$ belonging to some fixed, $\Delta t$-independent set. In fact, under the natural assumption that $c_{1,n}\Delta t$ is independent of $\Delta t$, it follows that
\begin{equation}\label{eq:qnorm}
\q(\eta)=q(\eta/\Delta t)\Delta t=\frac{1}{2\pi i}\sum_n ( c_{1,n}\Delta t)  e^{2\pi i n\eta } 
\end{equation}
is a trigonometric polynomial independent of $\Delta t$.

\begin{example}\label{ex:q}
Let $\DDe$ be given by \eqref{1order}, where the index $n$ ranges over the integers, and choose coefficients $c_{1,\pm 1}=\pm 1/(2\Delta t)$ and $c_{1,n}=0$ for all other values of $n$. Then $\DDe$ is the central difference operator
$$
\DDe v(t)=\frac{v(t+\Delta t)-v(t-\Delta t)}{2\Delta t},
$$
and $q(\omega)=\sin(2\pi\omega\Delta t)/2\pi\Delta t$. It follows that $q$ is invertible for $\omega\in\Omega$ where $\Omega=[-\frac{1}{4\Delta t},\frac{1}{4\Delta t}]$. In other words, $q$ is invertible when the normalized variable $\omega\Delta t$ satisfies $\lvert\omega\Delta t\rvert\le1/4$. Moreover, $\q(\eta)=\sin(2\pi\eta)/2\pi$.
\end{example}

\begin{example}\label{ex:2ndorder}
Let $\DDe^2$ be the approximation of a second order derivative given by
\begin{equation}\label{eq:2ndorder}
\DDe^2v(t)=\frac{v(t+\Delta t)+v(t-\Delta t)-2v(t)}{\Delta t^2},
\end{equation}
compare with the approximation of $\partial_x^2$ in Example \ref{ex:waveq}. Then $\DDe^2$ is equivalent to two applications of the central difference operator from Example \ref{ex:q} at half the step size, namely
$$
\DDe v(t)=\frac{v(t+\Delta t/2)-v(t-\Delta t/2)}{\Delta t}.
$$
It follows that $q(\omega)=\sin(\pi\omega\Delta t)/\pi\Delta t$ for $\omega\in\Omega=[-\frac{1}{2\Delta t},\frac{1}{2\Delta t}]$ and $\q(\eta)=\sin(\pi\eta)/\pi$. This finite difference operator appears in connection with certain leap\-frog schemes.
We mention that $q$ can also be found by applying a Fourier transform to both sides of \eqref{eq:2ndorder} followed by easy calculations, compare with Remark \ref{rmk:vonNeumann} below.
\end{example}

\begin{remark}\label{rmk:vonNeumann}
As mentioned in the introduction, there is a connection between the phase shift function and the so-called amplification factor appearing in von Neumann stability analysis; in fact, their definitions use the same idea although it is applied to the time domain for the phase shift function whereas it is applied to the spatial domain for the amplification factor. To see this, consider for example the one-dimensional heat equation $\partial_t u(t,x)=\partial_x^2u(t,x)$. Suppose we discretize the equation as
\begin{equation}\label{eq:discreteheat}
\frac{u(t+\Delta t,x)-u(t,x)}{\Delta t}=\frac{u(t,x+\Delta x)-2u(t,x)+u(t,x+\Delta x)}{\Delta x^2}
\end{equation}
where the time derivative is approximated using a forward Euler scheme and the second order spatial derivative is approximated using the finite difference scheme in Example \ref{ex:2ndorder} but now in $x$ instead of $t$. 
By taking a partial Fourier transform in $x$ of both sides we may write this as 
$$
\frac{\mathcal F_x(u(t+\Delta t,\cdot))(\xi)-\mathcal F_x(u(t,\cdot))(\xi)}{\Delta t}=\frac{e^{2\pi i\xi\Delta x}-2+e^{-2\pi i\xi\Delta x}}{\Delta x^2}\mathcal F_x(u(t,\cdot))(\xi).
$$
Rearranging terms and using Euler's formula gives
$$
\mathcal F_x(u(t+\Delta t,\cdot))(\xi)=\bigg[1+2\frac{\Delta t}{\Delta x^2}(\cos(2\pi \xi\Delta x)-1)\bigg]\mathcal F_x(u(t,\cdot))(\xi).
$$
The function in brackets is called the {\it amplification factor} at wave number $\xi$ and it depends on the choice of spatial discretization scheme. By the double angle formula we see that for this particular choice it satisfies
$$
1+2\frac{\Delta t}{\Delta x^2}(\cos(2\pi \xi\Delta x)-1)=1+\Delta t(2\pi i q(\xi))^2,
$$
where $q(\xi)=\sin(\pi\xi\Delta x)/\pi\Delta x$ is the phase shift function from Example \ref{ex:2ndorder}, now defined with respect to $\Delta x$ and evaluated at wave number $\xi$ instead of $\Delta t$ and frequency $\omega$.
(In von Neumann analysis one usually considers the case of a plane wave $u(t,x)=e^{2\pi ix\xi}$ at fixed time $t$ and assumes that $u(t+\Delta t,x)=g(\xi)u(t,x)$ for some amplification factor $g$ which is found by inserting these expressions in \eqref{eq:discreteheat}, see e.g., LeVeque \cite[Ch.~9]{leveque2007finite}; the result is the same.)
\end{remark}

\subsection{Time dispersion transforms}\label{ss:transforms}
Let $\mathbf 1_\Omega$ denote the characteristic function of a set $\Omega$, so that $\mathbf 1_ \Omega(\omega)=1$ if $\omega\in \Omega$ and $\mathbf 1_ \Omega(\omega)=0$ if $\omega\notin \Omega$.
Based on the previous discussion we will henceforth assume that the function $q$ introduced above is restricted to the largest subset $\Omega=\Omega(\Delta t)$ of its domain of definition containing the origin where the mapping $q:\Omega\to q(\Omega)$ is invertible. The inverse function $q^{-1}$ is assumed to be defined on $q(\Omega)$. We also assume that $\Omega$ and $q(\Omega)$ both exhaust $\R$ in the limit as $\Delta t\to0$.

\begin{dfn}\label{def:transforms}
Let $f_i(t,x)$ be a function integrable in $t$. Given a finite difference operator $\DDe$, let $q$ be the corresponding phase shift function in \eqref{def:q}. Define the {\it forward time dispersion transform} (FTDT) of $f_i(t,x)$ as
\begin{equation}\label{FTDT}
\FT (f_i)(t,x)= \int_\Omega e^{2\pi i  t\omega}\widehat{f}_i(q(\omega),x)    \, d\omega.
\end{equation}
Define the {\it inverse time dispersion transform} (ITDT) of $f_i(t,x)$ by
\begin{equation}\label{ITDT}
\IFT (f_i)(t,x)=\int_{q(\Omega)} e^{2\pi i  t\omega}\widehat{f}_i(q^{-1}(\omega),x)    \, d\omega.
\end{equation}
\end{dfn}

The definition extends in the natural way to distributions with well-defined Fourier transforms which are integrable on $q(\Omega)$ and $\Omega$. For example, since the Dirac measure $\delta(t)$ has Fourier transform $\widehat\delta(\omega)\equiv1$ the FTDT of $\delta(t)$ is 
$$
\FT(\delta)(t)= \int_\Omega e^{2\pi i  t\omega}    \, d\omega=\mathcal F^{-1}(\mathbf 1_\Omega)(t).
$$

\begin{example}\label{ex:delta}
Let $q(\omega)=\sin(2\pi\omega\Delta t)/2\pi\Delta t$ for $\omega\in[-\frac{1}{4\Delta t},\frac{1}{4\Delta t}]$, so that $q$ is the phase shift function corresponding to the finite difference operator in Example \ref{ex:q}. 
Then 
$$
\FT(\delta)(t)=\frac{1}{2\Delta t}\sinc(\pi t/2\Delta t),
$$
where $\sinc(t)=\sin(t)/t$ is the sinc function. 
\end{example}

For future purposes we record the fact that
\begin{equation}\label{ITDTalt}
\IFT(f_i)(t,x)=\int_{\Omega}e^{2\pi i tq(\omega)}\widehat{f}_i(\omega,x)q'(\omega)\, d\omega,
\end{equation}
which follows by a straightforward change of variable. Similarly, 
we also have
\begin{equation}\label{FTDTalt}
\FT(f_i)(t,x)=\int_{q(\Omega)}e^{2\pi i tq^{-1}(\omega)}\widehat{f}_i(\omega,x)\frac{1}{q'(q^{-1}(\omega))}\, d\omega.
\end{equation}
Finally, note that  
$$
\FT(f_i)(t,x)=\mathcal F^{-1}\Big[\mathbf 1_\Omega(\cdot)\widehat f_i(q(\cdot),x)\Big](t),
$$
which together with a straightforward calculation shows that
\begin{equation}\label{transformscomposed}
\IFT (\FT(f_i))(t,x)=\int_{q(\Omega)}e^{2\pi i  t\omega} \widehat{f}_i(\omega,x)    \, d\omega.
\end{equation}
In other words, $\IFT (\FT(f_i))$ does not equal $f_i$, but the bandlimited version of $f_i$ with frequency support contained in the range of $q$. Thus, if $f_i$ is already bandlimited then $\IFT(\FT(f_i))=f_i$ for sufficiently small $\Delta t$. As we will demonstrate, the effects of the approximation $\IFT (\FT(f_i))\approx f_i$ are negligible also for non-bandlimited functions with sufficient decay at infinity as long as $\Delta t$ is chosen sufficiently small (so that $q(\Omega)$ is large enough). This latter situation is analyzed in depth in what follows, in particular in Section \ref{section:numsim}.

\begin{example}
Consider again the case when $q(\omega)=\sin(2\pi\omega\Delta t)/2\pi\Delta t$ for $\omega\in[-\frac{1}{4\Delta t},\frac{1}{4\Delta t}]$.
Let $f$ be a bandlimited function so that $f(t)=\mathcal F^{-1}(\mathbf 1_{[-B,B]}\widehat f)(t)$ for some minimal number $B$ (the bandwidth). By the Nyquist-Shannon sampling theorem, the sampling rate necessary to accurately represent $f$ is $f_s>2B$. However, in order to utilize the entire frequency content of $f$ when computing the forward dispersion transform $\FT(f)$, the sampling rate has to be doubled since $[-B,B]\subset \Omega$ if and only if 
$B<\tfrac{1}{4 \Delta t}$, i.e., $f_s=1/\Delta t>4B$.
Furthermore, the sampling rate has to be effectively tripled in order for $\IFT(\FT(f))$ to equal $f$, since $[-B,B]\subset q(\Omega)$ if and only if 
$$
B<q(1/4 \Delta t)=\frac{1}{2\pi \Delta t},
$$
i.e., $f_s=1/\Delta t>2B\pi$.
These drawbacks can sometimes be removed, respectively improved, by using a staggered grid provided that the original equation \eqref{evo} permits such leapfrog discretization schemes (compare with Example \ref{ex:2ndorder}). See Section \ref{section:viscoelastic} and Appendix \ref{3dvisco} for an example of such an implementation.
\end{example}

We shall now examine the applications of Definition \ref{def:transforms} for evolution equation modeling. The following theorem establishes a correspondence between the system of differential equations \eqref{evo} and its counterpart obtained by approximating time derivatives with finite differences. In particular, it shows how to use the FTDT to modify the source function in \eqref{evo} when changing to a system of finite difference equations, and how the ITDT turns a function satisfying said finite difference system into an exact solution of the original evolution equation.

\begin{theorem}\label{thm:evo}
Let $u=(u_1,\ldots,u_K)$ be a solution to the evolution equation \eqref{evo}. Set $g_i=\FT(f_i)$ and $v_i=\FT(u_i)$. Then $v=(v_1,\ldots,v_K)$ satisfies the 
finite difference system
\begin{equation}\label{FDevo}
\PP_i v(t,x) = g_i(t,x),\quad 1\le i\le K,
\end{equation}
for each value of $t$, where $\PP_i$ is the finite difference operator given by \eqref{eq:defofFDsystem}, obtained from $\mathscr P_i$ by approximating time derivatives with finite differences. Conversely, suppose that $v=(v_1,\ldots,v_K)$ is a function satisfying \eqref{FDevo} for all $t$, where each $g_i$ and $v_i$ is integrable in $t\in\R$. Set $f_i=\IFT(g_i)$ and $u_i=\IFT(v_i)$. Then $u=(u_1,\ldots,u_K)$ satisfies \eqref{evo} for all $t$.
\end{theorem}

Note that in neither of the two statements of the theorem is the function $v$ obtained by {\it solving} the finite difference system \eqref{FDevo}, which comes with considerations of stability when choosing step size $\Delta t$. Obtaining $v$ by numerically solving \eqref{FDevo} is of course the ultimate goal, but this first requires a discussion about discrete versions of the dispersion transforms and is postponed until subsection \ref{ss:discreteITDT}. There we also address the issue of interpolating a discrete solution to make it possible to verify that it satisfies the continuous equation \eqref{evo}, see Theorem \ref{thm:inverse}.

\begin{proof}
First note that applying the Fourier transform to \eqref{evo} and evaluating at $q(\omega)$ gives
\begin{equation}\label{Fourierevo}
(2\pi iq(\omega))^{n_i} \widehat{u}_i(q(\omega),x)+\sum_{j=0}^{n_i-1}\sum_{k=1}^K(2\pi iq(\omega))^j L_{ijk}\widehat{u}_k(q(\omega),x)=\widehat{f}_i(q(\omega),x).
\end{equation}Next, using the definition of $\PP_i$ together with \eqref{FourierDj}--\eqref{def:q} we get
\begin{align*}
\PP_i v(t,x) &= \infint \bigg[ (2\pi i q(\omega))^{n_i}\widehat{v}_i(\omega,x) +\sum_{j=0}^{n_i-1}\sum_{k=1}^K(2\pi i q(\omega))^j L_{ijk}\widehat{v}_k(\omega,x)  \bigg]  e^{2\pi i t\omega} \, d \omega
\end{align*}
by the Fourier inversion formula.
Now substitute $\widehat{v}_i(\omega,x)=\mathbf 1_{\Omega} (\omega) \widehat{u}_i(q(\omega),x)$ and use \eqref{Fourierevo} to obtain $\PP_i v_i(t,x) 
=\int_\Omega  \widehat{f}_i(q(\omega),x)  e^{2\pi i t\omega} \, d \omega$. By construction, the right-hand side equals $g_i(t,x)$, which proves the first part of the theorem.

To prove the converse statement, suppose that $v$ is some function satisfying \eqref{FDevo} for all $t\in\R$. Since the $v_i$ are integrable we may apply the Fourier transform to both sides of \eqref{FDevo}. Doing so we find by inspecting \eqref{eq:defofFDsystem} and using \eqref{FourierDj}--\eqref{def:q} that
\begin{equation}\label{FDevoFourier}
(2\pi iq(\omega))^{n_i}\widehat v_i(\omega,x)+\sum_{j=0}^{n_i-1}\sum_{k=0}^K(2\pi iq(\omega))^jL_{ijk}\widehat v_k(\omega,x)=\widehat g_i(\omega,x).
\end{equation}
Setting $u_i=\IFT(v_i)$ we see by \eqref{ITDTalt} that $u_i(t,x)=\int_\Omega e^{2\pi i tq(\omega)}\widehat v_i(\omega,x)q'(\omega)\,d\omega$. Applying $\mathscr P_i$ to $u=(u_1,\ldots,u_K)$ and differentiating under the integral sign shows that $\mathscr P_iu(t,x)$ is equal to
$$
\int_\Omega e^{2\pi itq(\omega)}\bigg[(2\pi iq(\omega))^{n_i}\widehat v_i(\omega,x)+\sum_{j=0}^{n_i-1}\sum_{k=0}^K(2\pi iq(\omega))^jL_{ijk}\widehat v_k(\omega,x)\bigg]q'(\omega)\,d\omega.
$$
By \eqref{FDevoFourier} we conclude that $\mathscr P_iu(t,x)=\int_\Omega e^{2\pi itq(\omega)}\widehat g_i(\omega,x)q'(\omega)\,d\omega=\IFT(g_i)(t,x)$, where the last identity follows by \eqref{ITDTalt}. This completes the proof.
\end{proof}

We conclude this subsection with a few general remarks.

\begin{rmk1}\label{rmk:FIO} Close inspection of \eqref{ITDTalt} and \eqref{FTDTalt} using the normalized phase shift function $\q(\eta)$ defined in \eqref{eq:qnorm} shows that the dispersion transforms $\IFT$ and $\FT$ can be formally interpreted as Fourier integral operators depending on a small semiclassical parameter $h=\Delta t$ (see Appendix \ref{ss:FIO}). As such, they are associated with a canonical map $\chi$ and its inverse $\chi^{-1}$ acting on phase space via
\begin{equation}\label{eq:chi}
\chi:(t\q'(\eta),\eta)\mapsto(t,\q(\eta)).
\end{equation}
The physical meaning of this is well understood in terms of dynamics of wave packets \cite{faure2011semiclassical}. We provide a detailed presentation in 
\S\ref{ss:dynamics}, briefly summarized as follows: let $(t_0,\eta_0)$ be a point in phase space and consider a {\it Gaussian wave packet} defined by
$$
\varphi_{(t_0,\eta_0)}(t)=(2/\Delta t)^{\frac{1}{4}}
e^{2\pi i (t-t_0)\eta_0/\Delta t}e^{-\pi (t-t_0)^2/\Delta t}.
$$
When $t\ne t_0$ we have $\varphi_{(t_0,\eta_0)}(t)=O(\Delta t^\infty)$ as $\Delta t\to0$, where $O(\Delta t^\infty)$ means $O(\Delta t^N)$ for all $N>0$. Similarly, the semiclassical (i.e., scaled) Fourier transform 
\begin{align}\label{eq:dt-scaling}
\mathcal F_{\Delta t}(\varphi_{(t_0,\eta_0)} )(\eta)&=(\Delta t)^{-\frac{1}{2}}\mathcal F(\varphi_{(t_0,\eta_0)} )(\eta/\Delta t)\\&=
(2/\Delta t)^{\frac{1}{4}}e^{-2\pi i t_0\omega/\Delta t}e^{-\pi (\eta-\eta_0)^2/\Delta t}
\notag
\end{align}
is $O(\Delta t^\infty)$ as $\Delta t\to0$ if $\eta\ne\eta_0$. Such a function is said to be {\it microlocally small} outside $\{(t_0,\eta_0)\}$. By Proposition \ref{prop:wavefrontset}, the  ITDT of $\varphi_{(t_0,\eta_0)}$ is microlocally small outside 
$$
\{\chi(t_0,\eta_0)\}=\{(t_0/\q'(\eta_0),\q(\eta_0))\}
$$
and the FTDT of $\varphi_{(t_0,\eta_0)}$ is microlocally small outside
$$
\{\chi^{-1}(t_0,\eta_0)\}=\{(t_0\q'(\q^{-1}(\eta_0)),\q^{-1}(\eta_0))\}.
$$
Thus in this sense, as $\Delta t\to0$, the ITDT of $\varphi_{(t_0,\eta_0)}$ behaves like the wave packet $\varphi_{\chi(t_0,\eta_0)}$ and the FTDT of $\varphi_{(t_0,\eta_0)}$ behaves like the wave packet $\varphi_{\chi^{-1}(t_0,\eta_0)}$.
This phenomenon is illustrated in Figure \ref{fig:wavepacket}.

We also mention that by using arguments similar to those in the proof of Theorem \ref{thm:evo}, it is straightforward to check that $\FT\mathscr P_i =\PP_i\FT$. Viewing the dispersion transforms as Fourier integral operators, the proof of the first statement in Theorem \ref{thm:evo} would then proceed by simply noting that, by assumption, $v_i=\FT u_i$ and $g_i=\FT f_i$, so
$$
\PP_i v_i=\PP_i \FT u_i=\FT \mathscr P_i u_i=\FT f_i=g_i.
$$
The converse statement of Theorem \ref{thm:evo} can be proved in a similar way.
In the sequel we shall continue to prefer elementary proofs using explicit formulas instead of relying on the framework of microlocal analysis. However, this interpretation does succinctly highlight the obstruction caused by allowing time-dependent coefficients in \eqref{evo}, see the discussion in Appendix \ref{ss:FIO}.
\end{rmk1}

\begin{figure}[!t]
\centering
\includegraphics{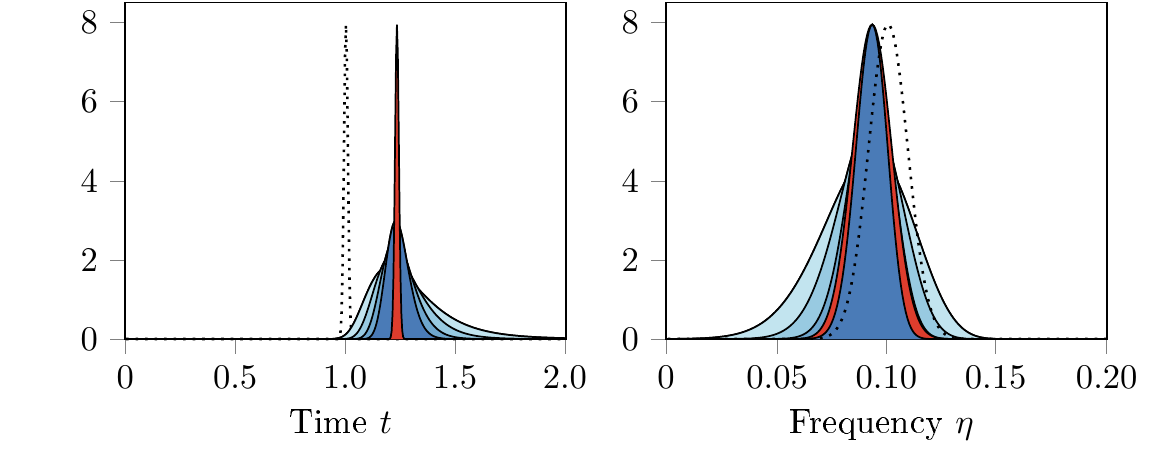}
\\
\includegraphics{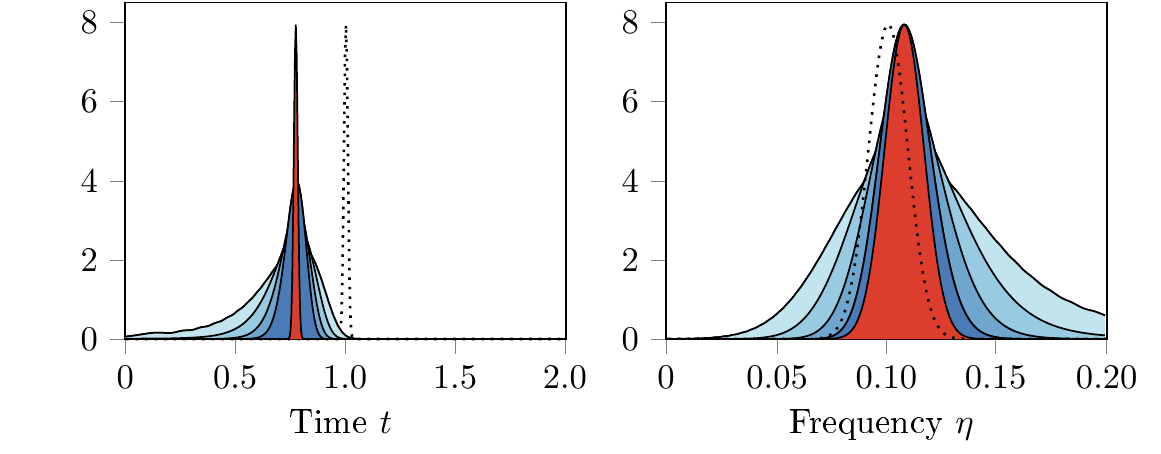}
\caption{The inverse (top) and forward (bottom) dispersion transforms of a wave packet $\varphi_{(t_0,\eta_0)}$ in the time domain (left) together with their scaled Fourier transforms (right, cf.~\eqref{eq:dt-scaling}), represented in blue by their corresponding modulus. Here $t_0=1$ and $\eta_0=0.1$, the normalized phase shift function is $\q(\eta)=\sin(2\pi\eta)/2\pi$, and the ITDT and FTDT of $\varphi_{(t_0,\eta_0)}$ are computed for $\Delta t$ ranging from $4\,\mathrm{ms}$ (light) to $0.5\, \mathrm{ms}$ (dark). In addition, $\varphi_{(t_0,\eta_0)}$ (dotted) as well as the wave packets microlocalized at $\chi(t_0,\eta_0)$ (red, top) and at $\chi^{-1}(t_0,\eta_0)$ (red, bottom) together with their scaled Fourier transforms are shown for $\Delta t=0.5\, \mathrm{ms}$ as reference. \label{fig:wavepacket}}
\end{figure}

\begin{rmk2} These are not mentioned in Theorem \ref{thm:evo}. Since the finite difference schemes we have considered so far are multistep methods, initial data for $t\le 0$ is required to get started. The natural choice is to impose the same initial conditions for \eqref{FDevo} as for \eqref{evo}; this is also motivated by the fact that for $t\le 0$, dispersion should not yet have started to affect the numerical solution. However, suppose that $u_i$ is the solution to \eqref{evo} with initial condition \eqref{init}, and let $v_i$ be a function satisfying \eqref{FDevo} with the same initial condition. Then $v(t,x)\equiv 0$ for $t\le0$, but due to the non-local nature of the dispersion transforms, this is not true for $\FT(u_i)$ which introduces an approximation error between $\FT(u_i)$ and $v_i$. On the other hand, according to Lemma \ref{lem:init} the error is small and controlled by the time-step size $\Delta t$ (see \S\ref{ss:init} for precise statements). 
Assuming that solutions of \eqref{FDevo} depend continuously on the initial data we thus conclude by Theorem \ref{thm:evo} that $\FT(u_i)$ will continue to stay close to $v_i$ for all time. 
The introduction of this error is also mitigated by the fact that when both dispersion transforms are used together in a modeling scenario as pre- and post-filters, then a reverse error is introduced during the post-filtering process. This is given credence by the numerical results in Sections \ref{section:numsim} and \ref{section:viscoelastic}. 
In view hereof we will from now on always assume that unless stated to the contrary, initial conditions are given by \eqref{init}.
\end{rmk2}

\begin{rmk3}\label{ss:nonreal} In this case, the phase shift function
$q$ will not be real-valued in general. Still, under certain conditions one can define a version of the FTDT and ITDT, although this requires sufficiently fast (exponential) decay of the solutions $t\mapsto u_i(t)$ for the definitions above to make sense. This happens, e.g., if $u_i(t,x)\equiv 0$ for $t<0$, $\lvert u_i(t,x)\rvert\le Ce^{-2\pi\alpha t}$ for some constants $C$ and $\alpha$, while $\im q(\omega)< \alpha$ for $\omega\in \Omega$, for then the integral 
$$
\infint e^{-2\pi i t q(\omega)}u_i(t,x)\, dt
$$
defining $\widehat u_i(q(\omega),x)$
is absolutely convergent for all $\omega\in \Omega$. In particular, $v_i(t,x)=\FT(u_i(\cdot,x))(t)$ is well defined. If the $f_i$ satisfy similar decay conditions then the first statement in
Theorem \ref{thm:evo} immediately generalizes to cover this situation. The converse statement can also be generalized using similar adjustments.
We leave it to the interested reader to fill in the details.
\end{rmk3}

\begin{rmk4}\label{rmk:nonmatching}
Due to the coupled nature of \eqref{evo} it was essential in the proof of Theorem \ref{thm:evo} that the same finite difference approximation of the time derivative was used for all involved operators $\mathscr P_i$. As soon as this is not the case, the result ceases to hold without appropriate modifications. For comparison, one such example of nonmatching finite difference schemes is provided in \S\ref{ss:nonmatching}.
\end{rmk4}

\subsection{Discrete transforms}\label{ss:discreteITDT}

Theorem \ref{thm:evo} shows how to use the FTDT to compensate for numerical dispersion when passing from a continuous equation to a finite difference equation, and that the ITDT should be used afterwards to turn a solution of the finite difference equation into a solution of the desired continuous equation. However, Theorem \ref{thm:evo} failed to address how to apply the ITDT to a solution obtained by actually solving a finite difference equation (modified using the FTDT). For this, we must introduce suitable discrete versions of the transforms. In the process, we will obtain a methodology for correctly simulating the solution to an evolution equation of type \eqref{evo}. 

We will demonstrate how to simulate the solution for $0\le t\le T$, where $T>0$ is the desired lifespan. Discretizing the equations in time and correcting for time dispersion leads us to solve the difference system \eqref{FDevo}. Suppose therefore that $v_i$ is a computed solution to \eqref{FDevo}, with known values $v_i(t_n,x)$ at times $t_n=n\Delta t$, $0\le n<T/\Delta t$. (We describe below how to compute the right-hand side of \eqref{FDevo} using discrete sums.) 
We assume that $T/\Delta t=N$ for some integer $N$, so that
$$
\Delta t=T/N,
$$
and denote by $S$ the set of sampling points
$$
S=\{n\Delta t:0\le n<N\}.
$$
We will assume that $\Delta t$ is chosen small enough depending on the spatial operators $L_{ijk}$ so that the resulting difference system \eqref{FDevo} is numerically stable.

We begin with a general discussion and let $f(t)$ be a function of $t\in[0,T]$ with known values at the points in $S$.
Let $\omega_m\in\Omega$ and introduce
$$
a_m(f)=\Delta t\sum_{n=0}^{N-1}f(n\Delta t)e^{-2\pi i n\Delta t\omega_m}. 
$$
This is a Riemann sum of the integral $\int_0^Tf(t)e^{-2\pi i t\omega_m}dt$ and as such an approximation of $\widehat f(\omega_m)$ provided $f$ vanishes outside $[0,T]$.
Inspecting the definition \eqref{ITDTalt} of $\IFT(f)$ we then choose a partition of $\Omega$ and define a function of the continuous variable $t\in[0,T]$ via 
\begin{equation}\label{IFTD}
\ITDT(f)(t)=\Delta \omega\sum_{\omega_m\in\Omega}a_{m}(f)e^{2\pi i q(\omega_m) t}q'(\omega_m).
\end{equation}
Here, $\Delta \omega$ is the distance between two consecutive points $\omega_{m+1}$ and $\omega_m$ in the partition. The formula is thus a Riemann sum of the integral $\int_\Omega \widehat f(\omega)e^{2\pi i q(\omega)t}q'(\omega)\, d\omega$. In view of \eqref{ITDTalt}, this is clearly a discrete representation of the ITDT defined in \S\ref{ss:transforms}. Its usage allows for modeling the desired solution of \eqref{evo} with correct evolution in time.

\begin{theorem}\label{thm:inverse}
Let $\Delta t=T/N$ and $v=(v_1,\ldots,v_K)$ be a solution of \eqref{FDevo} computed at times $t_n=n\Delta t$ for $0\le n<N$. Define $u_i(t,x)=\ITDT(v_i(\cdot,x))(t)$ and $f_i(t,x)=\ITDT(g_i(\cdot,x))(t)$. Then $u=(u_1,\ldots,u_K)$ solves \eqref{evo} for $0< t< T$.
\end{theorem}

\begin{proof}
In the proof we let $x$ be fixed and suppress it from the notation. If $f$ is a function sampled on $S$ and $\DDe$ is given by \eqref{1order} then a simple calculation shows that
\begin{equation}\label{findingq}
a_m(\DDe f)=a_m(f)\sum_jc_{1,j}e^{2\pi ij\omega_m\Delta t}.
\end{equation}
The second factor on the right is identified as $2\pi iq(\omega_m)$, with $q$ given by \eqref{def:q}. 
We record the fact that if $v_i$ solves \eqref{FDevo} then 
$a_m(\mathsf P_i v)=a_m(g_i)$,
which in view of \eqref{findingq} means that
\begin{equation}\label{Fcoefffs}
(2\pi iq(\omega_m))^{n_i}a_m(v_i)+\sum_{j=0}^{n_i-1}\sum_{k=1}^KL_{ijk}(2\pi iq(\omega_m))^{j}a_m(v_k)=a_m(g_i).
\end{equation}

Next, inserting the definition of 
$$
u_i(t)=\Delta \omega\sum_{\omega_m\in\Omega}a_m(v_i)e^{2\pi iq(\omega_m)t}q'(\omega_m)
$$
into \eqref{evo} and differentiating we get
\begin{multline*}
\mathscr P_iu(t,x)=\Delta \omega\sum_{\omega_m\in\Omega}\bigg[(2\pi iq(\omega_m))^{n_i}a_m(v_i)+\sum_{j=0}^{n_i-1}\sum_{k=1}^KL_{ijk}(2\pi iq(\omega_m))^{j}a_m(v_k)\bigg]\\ \times e^{2\pi i q(\omega_m)t}q'(\omega_m).
\end{multline*}
In view of \eqref{Fcoefffs} we conclude that
$$
\mathscr P_iu(t,x)=\Delta \omega\sum_{\omega_m\in\Omega}a_m(g_i)e^{2\pi i q(\omega_m)t}q'(\omega_m)=\ITDT(g_i)(t).
$$
By definition, the right-hand side is equal to $f_i(t)$, which completes the proof.
\end{proof}

Having verified that $\ITDT(v_i)(t)$ evolves correctly in time, we now discuss how the transform $\ITDT$ acts on arbitrary vectors in a discrete setting. 
 Given a solution $v_i$ to \eqref{FDevo} with known values $v_i(t_n,x)$ at times $t_n=n\Delta t$, $0\le n<T/\Delta t$, we first construct the function $\ITDT(v_i(\cdot,x))(t)$ as above. To obtain a function sampled on $S$ we simply evaluate $\ITDT(v_i(\cdot,x))(t)$ at the points $t=k\Delta t$, $k=0,\ldots,N-1$. 
This immediately generalizes to an arbitrary vector of length $N$: given any vector $(f_0,\ldots, f_{N-1})$ we define its inverse time dispersion transform by 
\begin{equation}\label{discreteITDT}
(\ITDT (f_n))_k=\Delta \omega\sum_{\omega_m\in\Omega}\Bigg(\Delta t\sum_{n=0}^{N-1}e^{-2\pi i n\omega_m\Delta t}f_n\Bigg)e^{2\pi i q(\omega_m)k\Delta t }q'(\omega_m),
\end{equation}
for $k=0,\ldots,N-1$.

We now describe how to compute the FTDT (of, e.g., the right-hand side of \eqref{evo}) using discrete sums. For any function $f$ sampled on $S$ we define a modified version of the samples $a_m(f)$ by
\begin{equation}\label{modFouriermode}
b_m(f)=\Delta t\sum_{n=0}^{N-1}f(n\Delta t)e^{-2\pi i n q(\omega_m)\Delta t},
\end{equation}
where the frequencies $\omega_m$ are as above. Thus $b_m(f)$ is defined by replacing $\omega_m$ by $q(\omega_m)$ in the definition of $a_m(f)$.
Next, define a function of the continuous variable $t\in[0,T]$ via
$$
\FTDT(f)(t)=\Delta \omega\sum_{\omega_m\in\Omega}b_m(f)e^{2\pi i \omega_m t},
$$
which in view of the previous discussion is a Riemann sum of the integral defining $\FT(f)(t)$. To obtain a function sampled on $S$ we evaluate $\FTDT(f)(t)$ at the points $t=k\Delta t$, $k=0,\ldots,N-1$. Finally, to define the FTDT of a vector we identify $f(n\Delta t)$, $n=0,\ldots, N-1$, with a vector $f=(f_0,\ldots, f_{N-1})$ and define the forward time dispersion transform of $(f_n)$ as 
$$
(\FTDT (f_n))_k=\FTDT(f)(k\Delta t),\quad k=0,\ldots, N-1.
$$
As with the inverse time dispersion transform, this immediately generalizes to an arbitrary vector of length $N$. Given any vector $(f_0,\ldots, f_{N-1})$ we thus define its forward time dispersion transform by
\begin{equation}\label{discreteFTDT}
(\FTDT (f_n))_k=\Delta \omega\sum_{\omega_m\in\Omega}\Bigg(\Delta t\sum_{n=0}^{N-1}e^{-2\pi i nq(\omega_m)\Delta t}f_n\Bigg)e^{2\pi i \omega_mk\Delta t }.
\end{equation}

Combined with Theorem \ref{thm:inverse}, the previous discussion yields the main result of this section.

\begin{theorem}\label{thm:mainthm}
Let $\Delta t=T/N$. Given a source function $f_i^\mathrm{orig}$ of \eqref{evo}, set $g_i=\FTDT(f_i^\mathrm{orig})$ and let $v=(v_1,\ldots,v_K)$ be a solution of \eqref{FDevo} computed at times $t_n=n\Delta t$ for $0\le n<N$. Define $u_i(t,x)=\ITDT(v_i(\cdot,x))(t)$. Then, for $0< t< T$, $u=(u_1,\ldots,u_K)$ solves \eqref{evo} with $f_i^\mathrm{orig}$ replaced by $\ITDT(\FTDT(f_i^\mathrm{orig}(\cdot,x)))(t)$.
\end{theorem}

We stress that, as mentioned in \S\ref{ss:transforms}, the composition of the FTDT and ITDT is not the identity mapping 
since $\IFT(\FT(f))$ is a bandlimited version of $f$ with frequency support contained in $q(\Omega)$. In particular, suppose we want to simulate a solution to \eqref{evo} with source term $f_i^\mathrm{orig}$. To do so, the method prescribed by Theorem \ref{thm:mainthm} is to compute $g_i=\FTDT(f_i^\mathrm{orig})$, i.e., the (discrete) FTDT of the source term, and solve the discretized equation \eqref{FDevo} with source term $g_i$. If $v_i$ is the obtained solution, the theorem implies that $u_i=\ITDT(v_i)$ simulates the evolution of the solution to the original equation \eqref{evo} but with source term
$$
f_i=\ITDT(g_i)=\ITDT(\FTDT(f_i^\mathrm{orig}))
$$
which is an approximation of the bandlimited version of $f_i^\mathrm{orig}$ with frequency support contained in $q(\Omega)$. 
If the frequency content of $f_i^\mathrm{orig}$ is negligible outside a compact set then one can make sure to capture its most relevant features by choosing $\Delta t$ sufficiently small.
This follows since $q(\Omega)\to\R$ as $\Delta t\to0$ and is investigated in detail in Section \ref{section:numsim} below (see Figure \ref{fig2}).

\begin{rmk}
Note that a priori, the vector $(f_n)$ in \eqref{discreteITDT} and \eqref{discreteFTDT} should be a vector representing a function sampled on $S$. If not, interpreting the continuous FTDT and ITDT as Riemann sums lead to different discrete formulas since the range of the index $n$ changes. Note also that although the evolution equation \eqref{evo} is translation invariant, the FTDT and ITDT transforms are not. In particular,  we have
$$
\FT(f)(t)\ne\FT(f(\cdot-t_0))(t+t_0),\quad \IFT(f)(t)\ne\IFT(f(\cdot-t_0))(t+t_0)
$$
in general. However, this is not a problem since we do in fact have
$$
\IFT(\FT(f))(t)=\IFT(\FT(f(\cdot-t_0)))(t+t_0)
$$
as a consequence of \eqref{transformscomposed} (in analogy with the Fourier inversion formula). Thus, when solving a Cauchy problem on, say, $[t_0,T+t_0]$, one can still apply \eqref{discreteITDT} and \eqref{discreteFTDT} to a vector representing a function sampled on $[t_0,T+t_0]$. Heuristically, this amounts to the same as translating the original equation to $[0,T]$, applying the transforms there, and translating back. In view of the discussion preceding this remark one is then simulating a solution to an evolution equation with source term
$$
f_i(t)=\ITDT(\FTDT(f_i^\mathrm{orig}(\cdot-t_0)))(t+t_0)=\ITDT(\FTDT(f_i^\mathrm{orig}))(t).
$$
\end{rmk}

\subsection{Fast implementation}\label{ss:fast}

In practice, the formulas \eqref{discreteITDT} and \eqref{discreteFTDT} can often be simplified once a specific choice of phase shift function $q$ is made. Specifically, 
\begin{itemize}
\item using the normalized variable $\omega\Delta t$ can allow the formulas to be interpreted as discrete Fourier transforms which can be implemented using the fast Fourier transform (FFT), and
\item using symmetry properties of $q$ and $\Omega$ can allow for more efficient algorithms.
\end{itemize}
Both situations are showcased in the following example.

\begin{example}\label{ex:implementation2}
Let $\DDe$ be the central difference operator from Example \ref{ex:q},
$$
\DDe v(t)=\frac{v(t+\Delta t)-v(t-\Delta t)}{2\Delta t}.
$$
Then $q(\omega)=\sin(2\pi\omega\Delta t)/2\pi\Delta t$. Assume as above that $\Delta t=T/N$ where $N$ is the number of sampling points in the time domain, and $T$ is the desired lifespan of the solution. 
Then $\Omega=\{\omega:\lvert\omega\rvert\le 1/4\Delta t\}$. To avoid cumbersome notation we will assume that $N$ is even so that $N/2$ is an integer. Inspecting \eqref{discreteITDT} we see that we can compute the inner sum by means of the discrete Fourier transform by choosing $\omega_m$ appropriately. We pick 
$$
\omega_m=\frac{1}{4\Delta t}\frac{2m}{N}
$$
so that $\omega_m\in\Omega$ when $m=-N/2,\ldots,N/2$. Substitution into \eqref{discreteITDT} gives after cancellations that
\begin{equation}\label{discreteITDTfast}
(\ITDT (f_n))_k=\frac{1}{2N}\sum_{m=-N/2}^{N/2}\Bigg(\sum_{n=0}^{N-1}e^{-2\pi i nm/2N}f_n\Bigg)e^{i k\sin(\pi m/N) }\cos(\pi m/N).
\end{equation}
As stated, the inner sum is the value at $m$ of the discrete Fourier transform of $\tilde f$, where $\tilde f$ is $f=(f_n)$ zeropadded to twice the length (i.e., the $m$:th Fourier mode of the vector $(f_0,\ldots,f_{N-1},0,\ldots,0)$ of length $2N$), and can be computed, e.g., using the FFT. The outer sum is the value at $k$ of a modified discrete inverse Fourier transform (truncated to use only the Fourier modes for $-N/2\le m\le N/2$ instead of the full range $-N\le m\le N-1$). If discrete transforms of numerous samples are to be computed, it is  advantageous to interpret \eqref{discreteITDTfast} as a linear map acting on the vector $(f_n)$ and compute the corresponding matrix. The cost of this operation scales as $O(N^2+N\log N)$.  Details for implementation in MATLAB can be found in Appendix \ref{ss:centralFD}.

In a similar manner we find by substituting the expression for $\omega_m$ into \eqref{discreteFTDT} that
\begin{equation}\label{discreteFTDTfast}
(\FTDT (f_n))_k=\frac{1}{2N}\sum_{m=-N/2}^{N/2}\Bigg(\sum_{n=0}^{N-1}e^{-i n\sin(\pi m/N) }f_n\Bigg)e^{2\pi i mk/2N }.
\end{equation}
Here, the inner sum is a modified discrete Fourier transform while the outer sum is a truncated discrete inverse Fourier transform at $k$. The outer sum can be computed using the inverse FFT. 
We also observe that if $f=(f_n)_{n=0}^{N-1}$ is a vector with real entries, then the inner sum in $(\FTDT (f_n))_k$ equals $b_m(f)/\Delta t$ in the notation above, where $b_{-m}(f)=\overline{b_{m}(f)}$ and bar denotes complex conjugation. This is a consequence of the fact that sine is an odd function. Similarly, $a_{-m}(f)=\overline{a_m(f)}$, and these symmetries can be used for a more efficient implementation. See Appendix \ref{app:matlab} for details.
Note that both \eqref{discreteITDTfast} and \eqref{discreteFTDTfast} only contain  frequency content up to a quarter of the sampling rate, i.e., up to {\it half} the Nyquist (folding) frequency. This situation is avoided when a leapfrog scheme can be employed, see Appendix \ref{ss:leapfrog}.
\end{example}

\begin{rmk} An alternative definition of $\ITDT$ found in the literature \cite{koene2018eliminating} is obtained by using Riemann sums to approximate \eqref{ITDT} instead of \eqref{ITDTalt}. One such example is
$$
(\mathcal{I}_\mathrm{disc}^\mathrm{alt}(f_n))_k=\Delta \xi\sum_{\xi_m\in q(\Omega)}\Bigg(\Delta t\sum_{n=0}^{N-1}e^{-2\pi i nq^{-1}(\xi_m)\Delta t}f_n\Bigg)e^{2\pi i \xi_mk\Delta t },
$$
where the $\xi_m$ are points evenly distributed in $q(\Omega)$ and the first factor is the distance between consecutive points $\xi_{m+1}$ and $\xi_m$. For implementation using the discrete Fourier transform, a natural option is to choose $\xi_m$ so that
$e^{2\pi i \xi_m k\Delta t}=e^{2\pi i mk/2N}$
for those $m$ for which $\xi_m\in q(\Omega)$. Then $\xi_m=\omega_m$ for $m$ in a subset of $[-N,N-1]$, and the formula above reduces to
$$
(\mathcal{I}_\mathrm{disc}^\mathrm{alt}(f_n))_k=\frac{1}{2N}\sum_{\xi_m\in q(\Omega)}\Bigg(\sum_{n=0}^{N-1}e^{-2\pi i nq^{-1}(\xi_m)\Delta t}f_n\Bigg)e^{2\pi i mk/2N}.
$$
(The absence of the factor $q'$ found in \eqref{discreteITDT} is explained by the relation 
$$
\Delta \xi=\xi_{m+1}-\xi_m=q(\widetilde \omega_{m+1})-q(\widetilde \omega_{m})\approx q'(\widetilde \omega_{m})(\widetilde \omega_{m+1}-\widetilde \omega_{m})=q'(\widetilde \omega_{m})\Delta\widetilde\omega,
$$
where $\widetilde\omega_m$ is the preimage of $\xi_m\in q(\Omega)$.)
It is easy to see that with $q$ as in Example \ref{ex:implementation2} this results in
$$
(\mathcal{I}_\mathrm{disc}^\mathrm{alt}(f_n))_k=\frac{1}{2N}\sum_{\lvert m\rvert\le M}\Bigg(\sum_{n=0}^{N-1}e^{- i n\arcsin(\pi m/N)}f_n\Bigg)e^{2\pi i mk/2N},
$$
where $M$ is the largest integer such that $M\le N/\pi$. Here, the inner sum is a modified discrete Fourier transform while the outer sum can be computed using the inverse FFT.
\end{rmk}

\section{Numerical simulations of a model equation}\label{section:numsim}

Here we propose to examine the accuracy of the method by solving a family of ordinary differential equations with known analytic solutions and comparing the resulting numerical solutions, corrected to account for dispersion, with simulations of the analytic expressions. To describe the method's inherent limitation that comes from restricting the frequency support of the adjusted source term (see the discussion after Theorem \ref{thm:mainthm}), we shall perform tests with source terms of varying frequency support. We consider the simple model
\begin{equation}\label{clinicalcauchyG}
\left\{\begin{aligned} u'(t)+u(t)&=f(t),\\
u(t)&\equiv 0,\quad t\le0,
\end{aligned}\right.
\end{equation}
where the source $f$ is a modulated Gaussian window function given by
\begin{equation}\label{eq:Gaussian}
f(t)=\frac{1}{\sqrt{2\pi\sigma^2}}\exp\left(-\frac{(t-\mu)^2}{2\sigma^2}\right)\times \exp\left (2\pi i a(t-\mu)\right).
\end{equation}
This is the probability density function of a normal distribution with mean $\mu$ and variance $\sigma^2$, modulated by the factor $\exp\left (2\pi i a(t-\mu)\right)$ with modulation parameter $a$ controlling the location of the frequency support of $f$.\footnote{In contrast to the wave packets discussed in the remark on page \pageref{rmk:FIO} and in Appendix \ref{ss:FIO}, the parameter $a$ is a priori independent of $\Delta t$. In addition, in line with the conventions of probability theory, the factor of normalization has been taken here with respect to the usual $L^1$ norm instead of the $L^2$ norm. } Because of the initial condition in \eqref{clinicalcauchyG}, the source term should vanish for $t\le0$. Moreover, in most applications that we have in mind, the energy is assumed to have dissipated by the end of the experiment. For these reasons we will center $f$ at (say) $t=5$ by taking $\mu=5$, take $\sigma$ so small that $f(t)$ is (practically) zero for $t\le0$, and run the experiment well past $t=10$. In particular, if $H$ is the Heaviside function then we will not distinguish between the functions $f(t)$ and $H(t)f(t)$ in what follows. Taking Fourier transforms we see that if $u$ is a solution of \eqref{clinicalcauchyG} then
$$
\widehat u(\omega)=\frac{1}{2\pi i\omega+1}\widehat f(\omega),
$$
where we identify the first factor on the right as the Fourier transform of $t\mapsto h(t)=e^{-t}H(t)$. By the Fourier inversion formula it follows that $u$ is given by the convolution
\begin{equation}\label{artificialsolution}
u(t)=h\ast f(t)=e^{-t}\int_{-\infty}^{t} e^{s}f(s)\, ds,\quad t>0.
\end{equation}
Applying the methodology presented in Section \ref{section:theory} we shall compare a sample of $u(t)$ for $0\le t\le 20$ with a numerically computed solution using the time dispersion transforms. To this end, we consider 
\begin{equation}\label{clinicalcauchyFD}
\left\{\begin{aligned} \DDe v+v&=g,\quad t>0,\\
v(t)&\equiv0,\quad t\le0,
\end{aligned}\right.
\end{equation}
where $\DDe$ is the central difference operator (appearing in Example \ref{ex:implementation2}) given by
\begin{equation*}
\DDe v(t)=\frac{v(t+\Delta t)-v(t-\Delta t)}{2\Delta t},
\end{equation*}
and $g=(g_k)$ is the FTDT of $f$, i.e.,
$$
g_k=(\FTDT(f))_k=\frac{1}{2N}\sum_{m=-N/2}^{N/2}\Bigg(\sum_{n=0}^{N-1}e^{-i n\sin(\pi m/N) }f(n\Delta t)\Bigg)e^{2\pi i mk/2N },
$$
compare with \eqref{discreteFTDTfast}. The sample $g$ is computed using the implementation of the FTDT described in Appendix \ref{ss:centralFD}. After solving \eqref{clinicalcauchyFD} we finally compute the ITDT of $v=(v_n)$ using formula \eqref{discreteITDTfast}, again implemented as described in Appendix \ref{ss:centralFD}. To minimize potential wraparound effects resulting from using the dispersion transforms (inherited from the FFT and the inverse FFT) on a modulated source function, we will solve the difference equation for $0\le t\le 20$ and apply a tapered cosine window, affecting the final sample points when $18\le t\le 20$, before computing the ITDT of the result. For transparency, we include plots obtained both with and without this taper.

\subsection{Varying the modulation}\label{ss:mod}
In Figure 
\ref{fig1}a we display the analytic solution $u(t)$ computed using \eqref{artificialsolution} and sampled at $t=n\Delta t$ with $\Delta t=0.02\, \mathrm{s}$. The source function $f$ was chosen to have mean $\mu=5\, \mathrm{s}$, variance $\sigma^2=0.1$ and modulation $a=0$. We furthermore show the numerical approximation of $u(t)$ and its error due to the standard central finite difference scheme and the forward Euler scheme. Finally, we use the time dispersion transform method to compute the solution, and show the difference between $u(t)$ and $\ITDT(v)(t)$ with and without using a taper. The numerical results are computed on a desktop with Intel Xeon CPU E5-1650 v3 $@$ $3.50\, \mathrm{GHz}$, running MATLAB 2017. It takes 0.083 seconds to compute and apply the FTDT to the source function; 0.00049 seconds for the 1000 time integration steps; and 0.077 seconds to compute and apply the ITDT to the solution vector. The time dispersion transform method outperforms the standard schemes by at least 9 orders of magnitude -- when the taper is used we even obtain accuracy up to an order of $10^{-15}$ on the range $t\in[0,18]$.

\begin{figure}[!t]
	\centering
		\includegraphics[scale=.88]{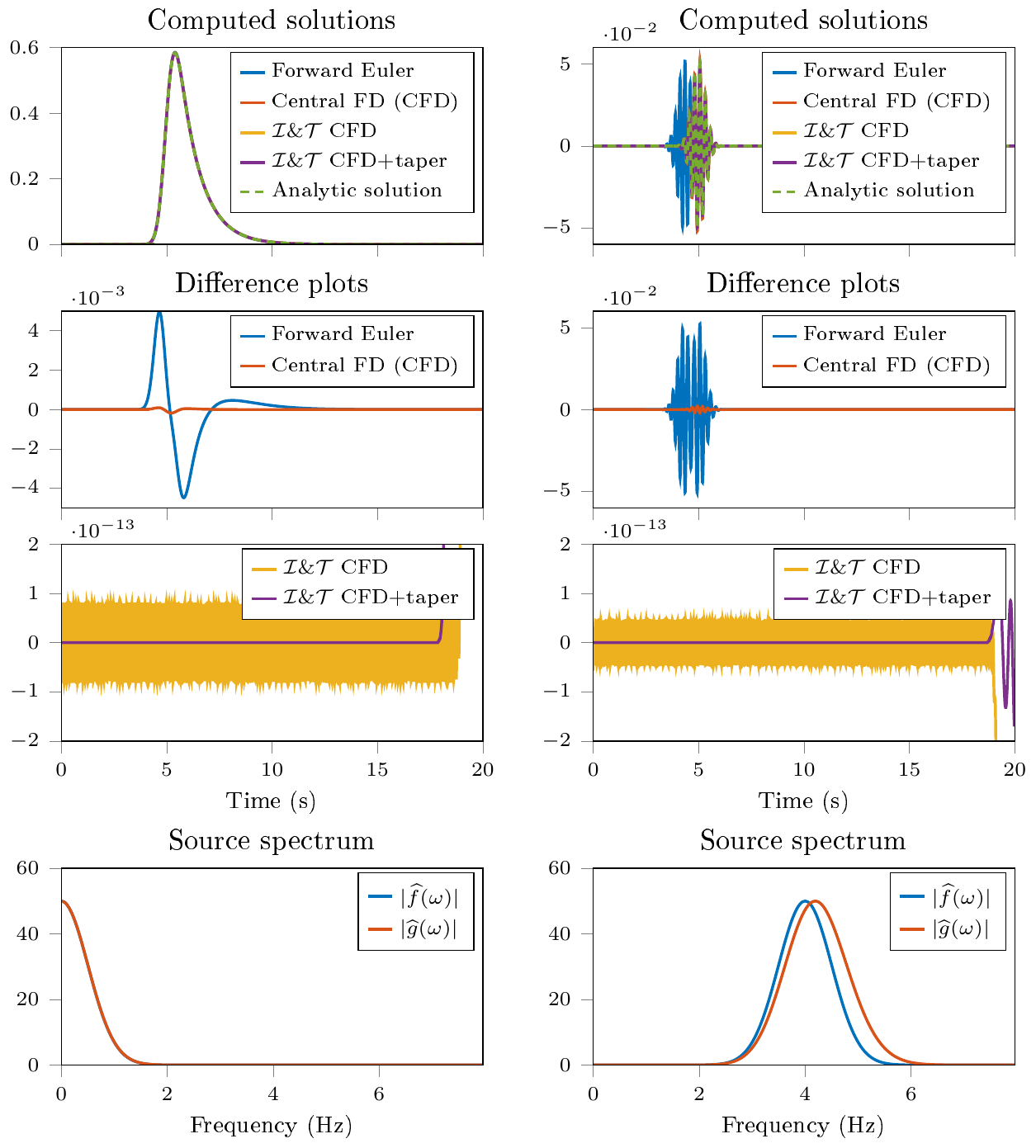}
		
		\hspace{0.45cm}\text{(a)}\hspace{5.45cm}\text{(b)}
\caption{Solving \eqref{clinicalcauchyG} numerically. The left (a) and right (b) panels correspond to an unmodulated and a modulated source, respectively. The difference plots show the difference between the analytic solution and the indicated numerical solutions. In the panels, Central FD (abbreviated CFD) refers to using a standard central finite difference scheme without applying the dispersion transforms, and is included for reference. $\IFT\&\FT$ CFD refers to the proposed method of using the forward time dispersion transform ($\FT$) to obtain the altered equation \eqref{clinicalcauchyFD}, solving this with the CFD, and applying the inverse time dispersion transform ($\IFT$) to the result. $\IFT\&\FT$ CFD+taper refers to doing the same with a taper added to the sample prior to applying the ITDT.}
\label{fig1}
\end{figure}

Figure 
\ref{fig1}b shows the result of adding a modulation by changing $a=0$ to $a=4$. We see that the Fourier support of the source function $f$ still sits comfortably within the critical frequency set $q(\Omega)$, which for $\Delta t=0.02\, \mathrm{s}$ is given by $q(\Omega)=\{\omega:\lvert \omega\rvert\le 25/\pi\}$ with $\omega$ measured in $\mathrm{Hz}$. The method continues to perform remarkably well, particularly in comparison with the forward Euler and central finite difference schemes. The computation time is identical to the previous case.

It should be mentioned that the discrete Fourier transform (and its implementation, the FFT) can be viewed as a trapezoidal rule quadrature applied to the Fourier transform integrals. For bandlimited functions, the approximation order is superalgebraic (meaning infinite approximation order) once the sampling is finer than what the bandlimitation prescribes. 
Implementing the transforms using the FFT and its inverse therefore makes the proposed method especially well suited for the type of essentially bandlimited source functions considered here and helps explain the high quality of these results.

\begin{figure}[!t]
	\centering
		\includegraphics[scale=.88]{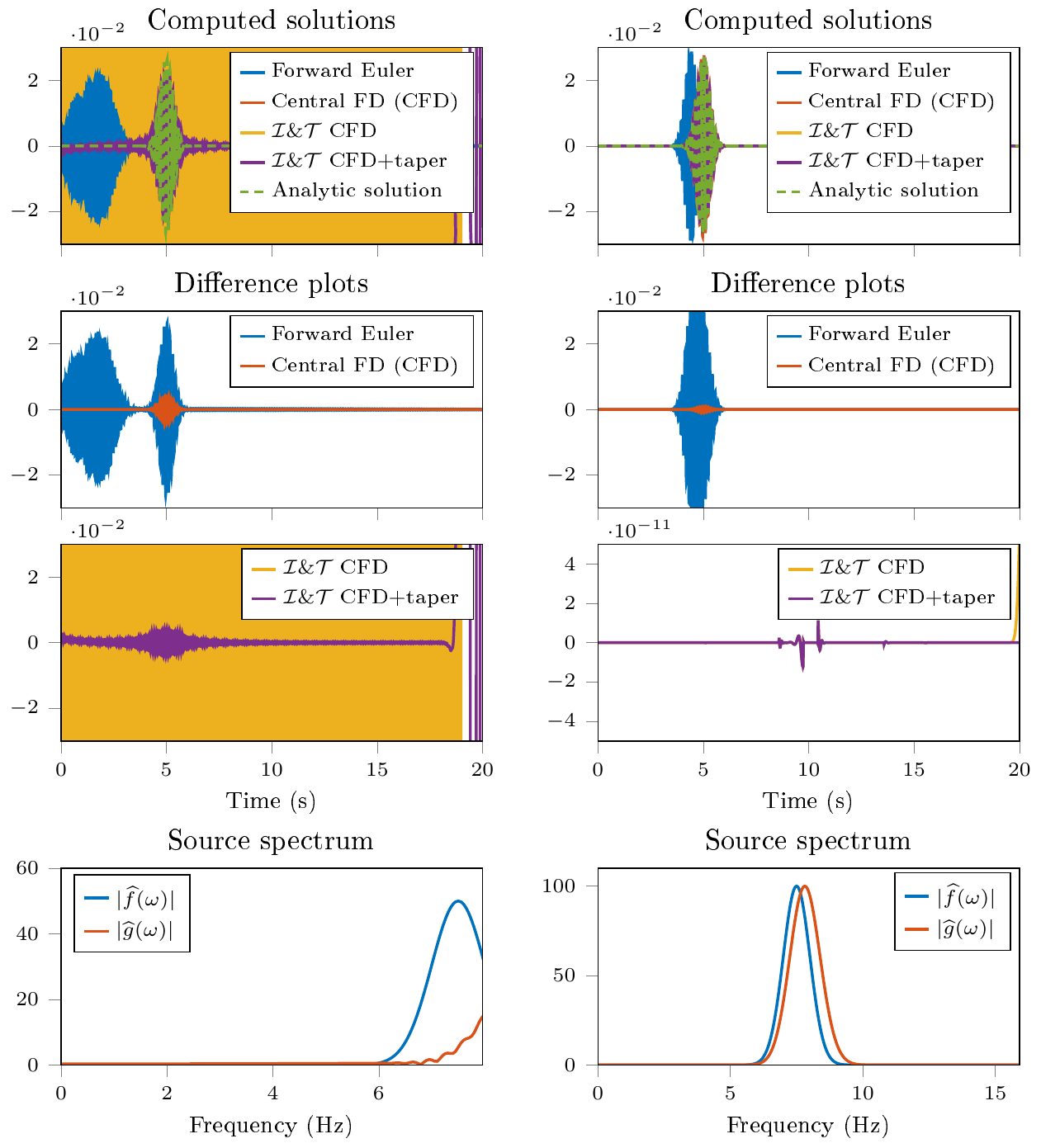}

		\hspace{0.45cm}\text{(a)}\hspace{5.45cm}\text{(b)}
\caption{Solving \eqref{clinicalcauchyG} numerically. The left panels (a) show the effect of a source with too large frequency support for the given spacing $\Delta t$. The right panels (b) show the effect of refining the spacing $\Delta t$ to accommodate. In the panels, Central FD (abbreviated CFD) refers to using a standard central finite difference scheme without applying the dispersion transforms, and is included for reference. $\IFT\&\FT$ CFD refers to the proposed method of using the forward time dispersion transform ($\FT$) to obtain the altered equation \eqref{clinicalcauchyFD}, solving this with the CFD, and applying the inverse time dispersion transform ($\IFT$) to the result. $\IFT\&\FT$ CFD+taper refers to doing the same with a taper added to the sample prior to applying the ITDT.}
\label{fig2}
\end{figure}

\subsection{Varying the frequency support}

In Figure 
\ref{fig2}a we have tried to break the method by setting $a=7.5$. We see that a part of the Fourier support of the source function $f$ is now outside the critical frequency set $q(\Omega)=\{\omega:\lvert \omega\rvert\le 25/\pi\}$ and the reconstruction of the analytic solution is quite poor. This is in part due to the strong oscillations of $f$; we see in Figure 
\ref{fig2}a that the Euler scheme also completely breaks down. However, we see that adding a taper results in partial recovery. The computation time is identical to the previous two cases.

As explained in the paragraph following \eqref{discreteFTDT}, we can improve the recovery by decreasing $\Delta t$, thus making sure that $q(\Omega)$ is large enough to capture the most relevant frequency content of $f$. The result of taking $\Delta t=0.01\, \mathrm{s}$ and keeping all other parameters the same can be seen in Figure 
\ref{fig2}b. Again, our proposed method performs at least 8 orders of magnitude better than the standard schemes. It takes 0.215 seconds to compute and apply the FTDT to the source function; 0.0175 seconds to compute the 2000 time integration steps; and 0.240 seconds to compute and apply the ITDT to the solution.

\section{Viscoelastic wave simulation}
\label{section:viscoelastic}

Here we test the accuracy of the method on seismic wave simulation, with the purpose of supplementing the previous section with a more realistic situation in which an analytic expression for the sought solution is not available. For comparison we provide simulations both of non-dissipative (elastic) as well as dissipative (viscoelastic) waves.

\subsection{The viscoelastic equations}\label{ss:memory}

A common approach to model wave propagation in anelastic media exhibiting both elastic and viscous behavior is to use viscoelastic theory \cite{robertsson1994viscoelastic}. For elastic media, it is common to use the analogy of a spring to model the medium under strain (by assuming that Hooke's law of proportionality between force and displacement holds). This is a linearized assumption that holds well for small displacements such as those found for seismic waves. In tensor notation, each component of the stress tensor will then be a linear combination of all components of the strain tensor. 
For viscoelastic media, the spring analogy is replaced by the analogy of a spring and dashpot in series, in parallel with another spring.
The resulting model is known as a {\it standard linear solid},
and several standard linear solids can be connected in
parallel to emulate a desired viscoelastic behavior. 
In a viscoelastic medium, then, each component of the stress tensor will be a linear combination of the entire history of the strain tensor, rather than only its current value. As it is memory-intensive to store the entire strain history, the system of equations is typically reformulated with the use of two constitutive relations, one that expresses the stress as a function of strain and a so-called memory variable, another that expresses the memory variable in terms of the strain and the memory variable itself. 
Below we recall the resulting equations in two and three dimensions; however our findings also apply to the one-dimensional case. For details of the derivation we refer to \cite{robertsson1994viscoelastic}, \cite{blanch1995modeling} and \cite{bohlen2002parallel}.

Wave propagation in a viscoelastic medium with $N$ standard linear solids can be described by Newton's second law
\begin{equation}\label{Newton2nd3D}
\rho\partial_tv_i=\partial_j\sigma_{ij}+f_i
\end{equation}
together with the stress-strain relation
\begin{align}\label{sigma}
\partial_t\sigma_{ij}&=\partial_kv_k\{\pi(1+\tau^p)-2\mu(1+\tau^s)\}+2\partial_jv_i\mu(1+\tau^p)\\&\quad+\frac{1}{N}\sum_{n=1}^Nr_{ij n}\quad \text{if }i=j,\notag\\
\partial_t\sigma_{ij}&=\big(\partial_jv_i+\partial_iv_j\big)\mu(1+\tau^s)+\frac{1}{N}\sum_{n=1}^Nr_{ij n}\quad \text{if }i\ne j,
\notag
\end{align}
with the so-called memory equations:
\begin{align}\label{memory}
\partial_t{r}_{ijn}&=-\frac{1}{\tau_{\sigma n}}\left\{ r_{ijn}+(\pi\tau^p-2\mu\tau^s)\partial_kv_k+2\partial_jv_i\mu\tau^s\right\}\quad\text{if }i=j,\\
\partial_t{r}_{ijn}&=-\frac{1}{\tau_{\sigma n}}\left\{ r_{ijn}+\mu\tau^s(\partial_jv_i+\partial_iv_j)\right\}\quad\text{if }i\ne j.\notag
\end{align}
In the equations above and throughout this section we employ Einstein notation and sum over repeated indices. The meaning of the symbols is as follows:
\begin{itemize}
\item[$\rho$] is the density,
\item[$\sigma_{ij}$] denotes the $ij$:th component of the stress tensor ($i,j=1,2$ in two dimensions and $i,j=1,2,3$ in three dimensions),
\item[$v_i$] denote the components of the particle velocity,
\item[$f_i$] are the components of external body force,
\item[$r_{ijn}$] are the $N$ memory variables $(n=1,\ldots,N$) corresponding to the stress tensor $\sigma_{ij}$,
\item[$\tau_{\sigma n}$] is the stress relaxation time of the $n$:th standard linear solid for both pressure and shear waves,
\item[$\tau^p$, $\tau^s$] define the level of dissipation for pressure and shear waves, respectively,
\item[$\pi$] denotes the relaxation modulus corresponding to pressure waves analogous to $\lambda+2\mu$ in the elastic case, where $\lambda$ and $\mu$ are the Lam{\'e} parameters,
\item[$\mu$] is the relaxation modulus corresponding to shear waves and is the analog of the Lam{\'e} parameter $\mu$ in the elastic case.
\end{itemize}
The equations describe the propagation of mechanical waves through a solid, in terms of strains (particle displacements away from their resting position) and stress disturbances (away from the reference stress states), as a function of spatially varying material properties. In geophysics, these equations can be used to model the propagation of seismic waves through realistic dissipative Earth models. The viscoelastic behavior is governed by the $2+N$ parameters $\tau^p$, $\tau^s$ and $\tau_{\sigma n}$, $1\le n\le N$, which all depend on the spatial variable $x$. Here $\tau^p$ and $\tau^s$ are computed using a variable so-called $Q$ model consisting of one component $Q_p$ for the pressure wave and one component $Q_s$ for the transverse shear wave via
\begin{equation}\label{taus}
\tau^p=2/Q_p,\quad \tau^s=2/Q_s.
\end{equation}
$Q$ is a quality factor that approximately measures the amount of energy dissipation \cite{bohlen2002parallel}, with $Q_p,Q_s\to\infty$ corresponding to the elastic, undamped case.

We end by noting that \eqref{Newton2nd3D}--\eqref{memory} constitutes a system of equations of the form \eqref{evo}. Indeed, considering the two-dimensional (2D) case, denote $v_1,v_2$ by $u_1,u_2$, denote the three distinct $\sigma_{ij}$, $i,j=1,2$, by $u_3,u_4,u_5$, and the $3N$ distinct memory variables $r_{ijn}$ by $u_{6},\ldots,u_{3N+5}$. Finally, we note by inspecting the equations above that the spatial operators involved are linear and independent of $t$.

\subsection{Model introduction}
We apply the theory of the previous section on a viscoelastic wave modeling example, using the leapfrog scheme described in detail in Appendix \ref{3dvisco} to solve \eqref{Newton2nd3D}--\eqref{memory}. Since leapfrog schemes are the simplest energy-conserving integrators \cite{hairer2006geometric} this is a natural choice in order to avoid numerical dissipation errors and thus isolate the effects caused by numerical dispersion errors.

We use the open-source 2D modeling engine SOFI2D developed by Bohlen {\it et al.}~\cite{SOFI2Dmanual} to perform the 2D viscoelastic simulation. Time is discretized into steps of constant length $\Delta t$. Similarly, continuous space is discretized into a 2D grid with spacing of $\Delta x$ and $\Delta z$ in the $x$ and $z$ directions. The wave equation is then solved using staggering of quantites in space, and using the leapfrog method to integrate the wave equation in time \cite{virieux1986p}. The spatial derivatives are efficiently approximated with a central 1D finite difference stencil of half-order 6:
\begin{equation*}
  \partial_x f(x_j) \approx \sum_{l=1}^6 \frac{\alpha_l}{\Delta x_j}\left( f(x_j+(l-\tfrac{1}{2})\Delta x_j)-f(x_j-(l-\tfrac{1}{2})\Delta x_j) \right),
\end{equation*}
for which the weights are given in Table \ref{tab:fdweights}.
\begin{table}[!ht]
\caption{The central finite difference weights used to compute the spatial first-order derivatives, truncated to 4 digits.}\label{tab:fdweights}
\begin{center}
\begin{tabular}{cccccc}\noalign{
\global\dimen1\arrayrulewidth
\global\arrayrulewidth1pt
}\hline
\noalign{
\global\arrayrulewidth\dimen1
}
$\alpha_1$&$\alpha_2$&$\alpha_3$&$\alpha_4$&$\alpha_5$&$\alpha_6$ \vspace{1pt} \\ \hline \vspace{-3mm}\\
1.2508 & $-$0.1203 & 0.0321 & $-$0.0101 & 0.0030 & $-$0.0007 \\  \noalign{
\global\dimen1\arrayrulewidth
\global\arrayrulewidth1pt
}\hline
\noalign{
\global\arrayrulewidth\dimen1
}
\end{tabular}
\end{center}
\end{table}
These weights correspond to an equiripple (minimax) filter that keeps the group-velocity error of the first-order derivative approximation confined to within 0.1\%. Such `optimal' finite difference coefficients are customary in geophysical finite difference modeling \cite{moczo2007}, see e.g.~\cite{kindelan1990construction} for the design procedure. We state the Courant-Friedrichs-Lewy (CFL) condition that ensures stability of the 2D simulation given the second-order accurate integration of the equations in time, as a function of the chosen discretizations and maximum velocity encountered in the simulation:
\begin{equation*}
  v_\text{max}\Delta t\sqrt{ \left(\sum_{l=1}^6\left|\frac{\alpha_l}{\Delta x}\right|\right)^2+\left(\sum_{l=1}^6\left|\frac{\alpha_l}{\Delta z}\right|\right)^2} \leq 1.
\end{equation*}
We will see that the maximum velocity present in the model is $4700\, \mathrm{m/s}$, and we choose a spatial discretization of $\Delta x=\Delta z=12.5\, \mathrm{m}$. The maximum stable time-step then follows as $\Delta t=1.3\, \mathrm{ms}$. We choose this as the `coarse' time-step. We can compare this coarse solution against an additional `fine' simulation, which uses a time-step of $\Delta t=0.013\, \mathrm{ms}$, which we consider to be the reference solution for our purposes.

The model used for the simulation is the Marmousi 2 model \cite{martin2002marmousi}, which provides a density model ($\rho$), a model for compressional wave velocities ($v_p$) and transverse shear velocities ($v_s$) which reflect the instantaneous elastic deformation modes for \eqref{Newton2nd3D}--\eqref{memory}:
\begin{align*}
  \pi  &=  \rho v_p^2,\\
  \mu&= \rho v_s^2.
\end{align*}
The models for $\rho$, $v_p$ and $v_s$ are shown in Figure \ref{fig:marmousimodel}a. For the viscoelastic modeling we furthermore create the $Q$ model used in \eqref{taus} by smoothing the $v_s$ and $v_p$ models and normalizing them to a maximum $Q$ of 350, as shown in Figure \ref{fig:marmousimodel}b. 
Additionally shown in the figures are the source location at $(x,z)=(0,25)$ and a series of recorders along the entire upper model boundary at $(x,z)=(n\cdot 12.5,62.5)$ for $n=-679,\dots,679$, with the coordinates in meters. One specific recorder at $(x,z)=(4625,62.5)$ is highlighted as an arbitrarily shown recorder that will be zoomed in upon in the results.
\begin{figure}[!t]
  \centering
  \includegraphics{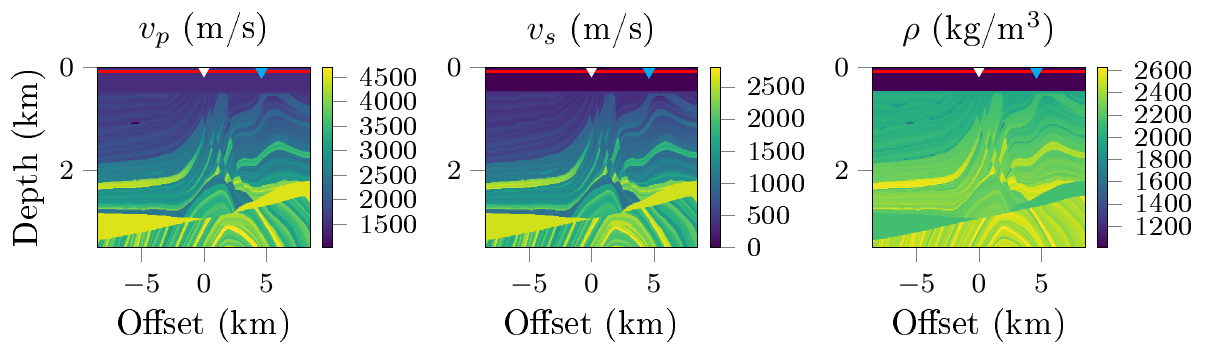}\\
  \text{(a)} The elastic Marmousi model.

  \includegraphics{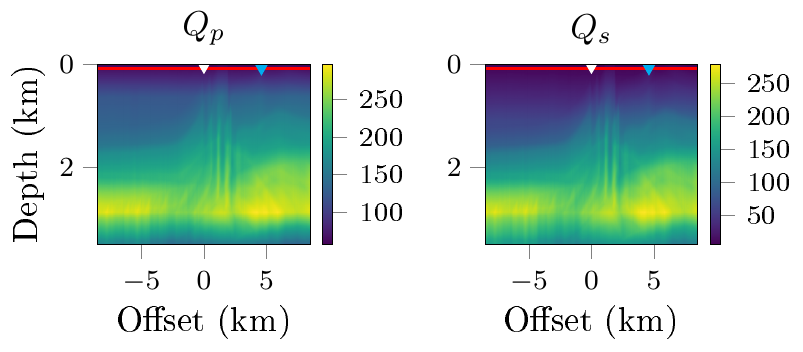}\\
  \text{(b)} The $Q$ model for the viscoelastic case.

  \caption{The compressional wave velocity ($v_p$), transverse shear wave velocity ($v_s$) and density ($\rho$) of the Marmousi 2 elastic model in \text{(a)}, and the $Q$ model for the compressional wave ($Q_p$) and the transverse shear wave ($Q_s$) in \text{(b)}. The source location is denoted by a white triangle in the top center of the models, the receiver line is denoted by the red line on the top of the models, a single recording station is denoted in the cyan triangle in the top-right of the models, which is used for the detailed zoom in Figures \ref{fig:elastic}--\ref{fig:viscoelastic}.}
  \label{fig:marmousimodel}
\end{figure}

The source-time function of the model is a typical seismic source wavelet, described as a $15\, \mathrm{Hz}$ peak frequency Ricker wavelet with a time-delay of $0.15\,\mathrm{s}$:
\begin{equation}\label{eq:rickerwavelet}
  f(t) = (1-2\pi^215^2(t-0.15)^2)e^{-\pi^215^2(t-0.15)^2}. 
\end{equation}
The source is injected as an explosive source that radiates equally in all directions. The recorders along the upper model boundary record the pressure variations (the diagonal stress components $\sigma_{xx}+\sigma_{zz}$) as a function of time at every $\Delta t$ simulated. 

\subsection{The elastic model results}
We first test the theory in an elastic case, in which we take $Q\to\infty$ so there is no damping, and with $N=0$ so there is no relaxation mechanism at all. The evolution of the wave equation is then computed from time 0 to $5\,\mathrm{s}$ with three model runs:
\begin{enumerate}
  \item using a coarse time-step of $\Delta t=1.3\,\mathrm{ms}$ without correcting the source or receiver time-series,
  \item using a coarse time-step of $\Delta t=1.3\,\mathrm{ms}$, but using the FTDT 
  to correct the source injection time-series \eqref{eq:rickerwavelet}, and the ITDT 
  to correct the recorded time-series,
  \item using a fine time-step of $\Delta t=0.013\,\mathrm{ms}$ as a reference solution.
\end{enumerate}
As the implemented wave equation solver scales linearly in time (keeping the spatial discretization $\Delta x=\Delta z=12.5\,\mathrm{m}$ in place), the third simulation thus costs 100 times more computational time. In this elastic instance, it takes 50 seconds to compute simulations (1) and (2), but 75 minutes for the simulation (3). Application in simulation (2) of the FTDT 
on the source-time function takes half a second to compute, and applying the ITDT 
to \emph{all} 1359 recorded signals takes seven seconds in total. This is, essentially, of negligible cost compared to the fine simulation (3).

\begin{figure}[!t]
  \centering
  \includegraphics{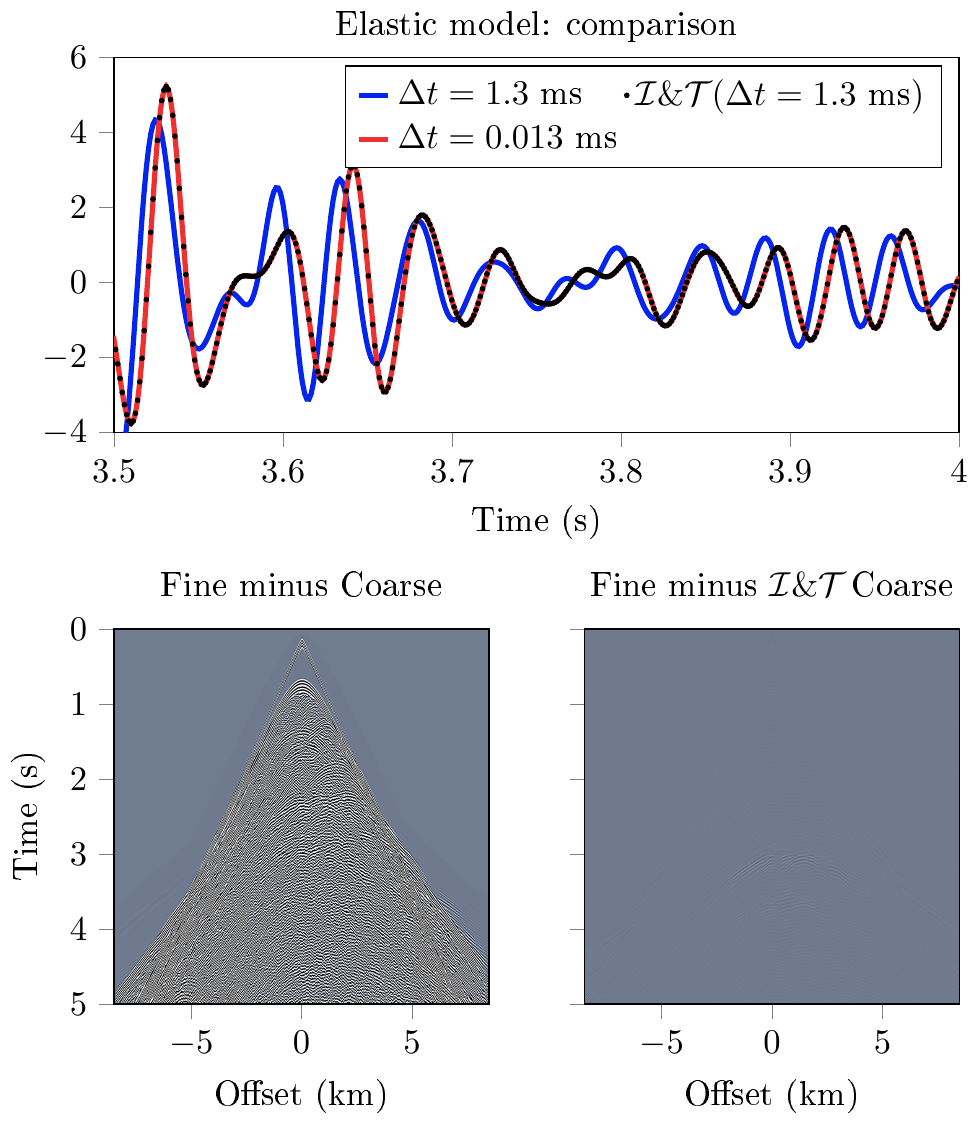}
  \caption{Removing dispersion for an elastic case. The top panel zooms in on a single recorder at $(4625,62.5)$ (a so-called `trace' at offset $4.625\,\mathrm{km}$). The graph in blue represents the coarse time-step solution, the graph in red the reference solution with a fine time-step. The graph in black dots uses the dispersion transforms to modify the source and recording of another coarse simulation as suggested in this paper. The bottom two panels show the difference between the three different data sets for all 1359 recordings at varying offset. }
  \label{fig:elastic}
\end{figure}

After finishing all the computations, we subsample the fine simulation to be able to compare the results sample-by-sample. The computed result is then shown in Figure \ref{fig:elastic}. We show a zoom on a single recording (its location is denoted in cyan in Figure \ref{fig:marmousimodel} but was chosen arbitrarily). It is clearly visible that the coarse simulation (1) created a recording that differs starkly from that made within the fine simulation (3). Conversely, after applying the time dispersion transforms on the source-time function and the receiver time-function, we obtain a solution that follows the correct phase and amplitude of the fine simulation. The two images below the graph in Figure \ref{fig:elastic} subtract the results of the coarse simulations from the fine simulation -- confirming that the correction procedure in this paper removes the dispersion error effectively for all recordings. The sum of all 1359 root mean square (RMS) errors along all traces is 1634 for the coarse case, and 1.6 for the simulation with the proposed time dispersion transforms -- the error energy is thus reduced by a factor 1018. The remaining errors seem to be of localized impact only, affecting strong peaks and throughs in the time-series, but do not seem to accumulate over time.

\begin{figure}[!t]
  \centering
\includegraphics{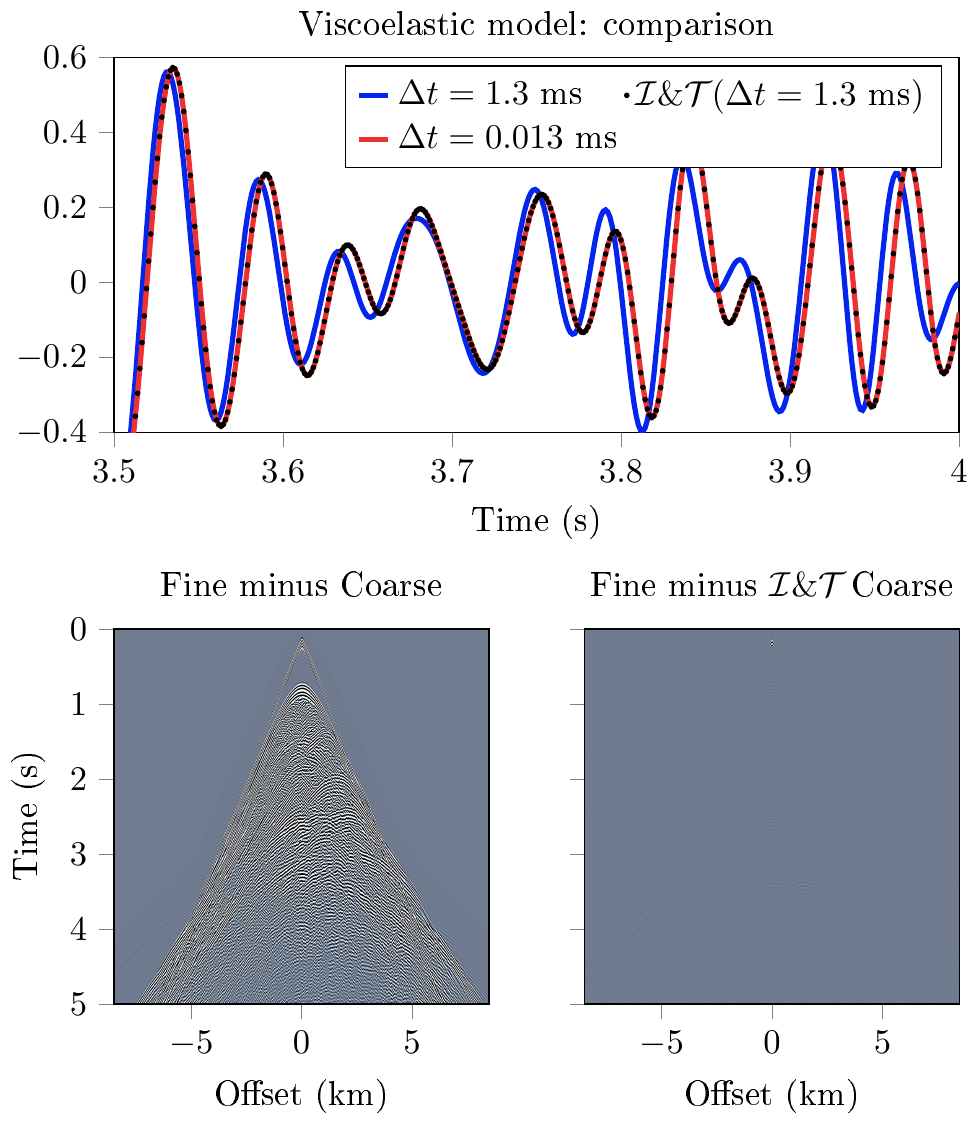}
  \caption{Removing dispersion for a viscoelastic case. The top of the figure zooms in on a single recorded trace at offset $4.625\,\mathrm{km}$. The graph in blue represents the coarse time-step solution, the graph in red the reference solution with a fine time-step. The graph in black dots uses the dispersion transforms to modify the source and recording of another coarse simulation as suggested in this paper. The bottom two figures show the difference between the three different data sets for all 1359 recordings at varying offset. }
  \label{fig:viscoelastic}
\end{figure}

\subsection{The viscoelastic model results}
The viscoelastic model uses three relaxation mechanisms ($N=3$) to model the spatially heterogeneous $Q$ model. Apart from these changes to the physical model, we proceed in exactly the same way as in the previous example. The computed result is displayed in Figure \ref{fig:viscoelastic}. The amplitudes in this model decrease with time as exemplified by the now 10 times smaller amplitudes in the graph compared to the elastic case, due to the damping. Like the elastic case before, the coarse simulation with $\Delta t=1.3\,\mathrm{ms}$ differs significantly from the fine simulation with $\Delta t=0.013\,\mathrm{ms}$. Conversely, applying the proposed corrections to the coarse simulation creates an adequate fit to the fine simulation. The computational time of all simulations is roughly doubled compared to the elastic simulation at 100 seconds for the coarse simulations and over 2 hours for the viscoelastic simulation. The FTDT and ITDT are still applied to the source-time function in half a second and to all 1359 traces in 7 seconds, respectively. The sum of all 1359 RMS errors along all traces is 212 for the coarse case, and 1.47 for the coarse case using the proposed transforms. Again, the error energy is reduced (now by a factor of 144) at very little additional cost. Again, these errors do not accumulate for longer simulation times.

\section*{Acknowledgements}
Jens Wittsten was supported by the Swedish Research Council grants 2015-03780 and 2019-04878. Erik Koene was supported by SNF grant 2-77220-15. 
We gratefully acknowledge the work of Thomas Bohlen, Denise De Nil, Daniel Köhn and Stefan Jetschny on the open-source code SOFI2D used in the viscoelastic simulations.  
We wish to thank Christof Stork for interesting discussions. We also thank the referee for valuable comments and suggestions.

\appendix

\renewcommand{\thesection}{\Alph{section}}
\renewcommand\thefigure{\arabic{figure}}

\section{Auxiliary results}

\subsection{Fourier integral operators}\label{ss:FIO}

Here we show how the dispersion transforms can be naturally understood as Fourier integral operators (FIO). In this context it will be convenient to view $\Delta t$ as a small (semiclassical) parameter $h>0$. We define the semiclassical Fourier transform of a function $f(t)$ by
$$
\mathcal F_h(f)(\eta)=\frac{1}{\sqrt h}\int e^{-2\pi i t\eta/h}f(t)\, dt,
$$
so that $\widehat f(\omega)=h^{\frac{1}{2}}\mathcal F_h(f)(\omega h)$. The Fourier inversion formula then takes the form
$$
f(t)=\frac{1}{\sqrt h}\int e^{2\pi i t\eta/h}\mathcal F_h(f)(\eta)\, d\eta.
$$
Standard references for semiclassical analysis are Martinez \cite{martinez2002introduction} and Zworski \cite{zworski2012semiclassical}.

Recall the normalized phase shift function $\q$ introduced in \eqref{eq:qnorm} which satisfies $\q(\omega h)/h=q(\omega)$, and define $\On$ by $\eta=\omega h\in\On$ if and only if $\omega\in\Omega$. By changing variables in \eqref{ITDTalt} it is easy to see that
\begin{equation}\label{semiclassicalI}
\IFT(f)(t)=\frac{1}{\sqrt h}\int_{\On}e^{2\pi i t\q(\eta )/h}  \q'(\eta )\mathcal F_h(f)(\eta )\, d\eta.
\end{equation}
Note that $q'(\omega)=\q'(\omega h)$ and $q^{-1}(\omega)=\q^{-1}(\omega h)/h$, which implies that $q'(q^{-1}(\omega))=\q'(\q^{-1}(\omega h))$, so making the change of variables $\eta=\omega h$ in \eqref{FTDT} similarly gives
\begin{equation}\label{semiclassicalT}
\FT(f)(t)=\frac{1}{\sqrt h}\int_{\q(\On)}e^{2\pi i t\q^{-1}(\eta)/h}\frac{1}{\q'(\q^{-1}(\eta))}\mathcal F_h(f)(\eta)\, d\eta.
\end{equation}

If $f$ is a function whose semiclassical Fourier transform has support contained in a set $U$ we will say that $f$ is $h$-bandlimited in $U$. Applying the ITDT to a function already $h$-bandlimited in $\On$ can be naturally understood, in view of \eqref{semiclassicalI}, as the action of a semiclassical FIO (call it $A$) which in $\R\times\On\subset T^\ast(\R)$ has phase function $\varphi(t,\eta)=t\q(\eta)$ and symbol $a(t,\eta)=\q'(\eta)$. $A$ is associated to the canonical transformation locally given by 
$$
\chi: (\varphi_\eta'(t,\eta),\eta)\mapsto (t,\varphi_t'(t,\eta)),
$$
i.e., by \eqref{eq:chi}.
Similarly,  $\FT$ is a semiclassical FIO (call it $B$) which in $\R\times \q(\On)$ has phase function $\psi(t,\eta)=t\q^{-1}(\eta)$ and symbol $b(t,\eta)=1/\q'(\q^{-1}(\eta))$. $B$ is associated to the inverse map $\chi^{-1}$. The composition $BA$ acts as the identity operator on functions $h$-bandlimited in $\On$. The composition $AB$ acts as the identity operator on functions $h$-bandlimited in $\q(\On)$.
Furthermore, using the Fourier inversion formula and arguments similar to those in the proof of Theorem \ref{thm:evo}, it is straightforward to check that $B\mathscr P_i =\PP_iB$, so that, as operators acting on functions $h$-bandlimited in $\On$, $\PP_i=B\mathscr P_i A$.

Note that the previous discussion can also be had in the framework of microlocal analysis for fixed $\Delta t$, i.e., without viewing $\Delta t$ as a semiclassical parameter. Our choice was made in preparation for \S\ref{ss:dynamics} below. If one instead takes the other viewpoint and repeats the arguments above one finds that the dispersion transforms are realized as FIOs associated to the canonical map
$$
\chi_q:(tq'(\omega),\omega)\mapsto(t,q(\omega))
$$
and its inverse. This gives a formula for the appropriate discrete operator to be used for given choice of discrete approximation $\DDe$ of the time derivative, even in the case when time-dependent coefficients are allowed in \eqref{evo}. In fact, if $q$ is the corresponding phase shift function, then the previous paragraph shows that $\mathscr P_i$ should be replaced by (a discretized version of)
$$
Q_i=\FT\mathscr P_i\IFT.
$$
By Egorov's theorem, this operator is a pseudodifferential operator with an integral representation
$$
Q_iv(t)=\int_\Omega e^{2\pi it\omega }Q_i(t,\omega)\hat v(\omega)\, d\omega,
$$
where, with abuse of notation, the symbol $Q_i(t,\omega)$ is a function defined on phase space (omitting all dependence on the spatial variable $x$). Assuming that $\mathscr P_i=p_{in_i}(t)\partial_t^{n_i}$ plus lower order terms, the principal symbol $\sigma(Q_i)$ of $Q_i$ is given by
$$
\sigma(\mathscr P_i)(\chi_q(t,\omega))=p_{in_i}(t/q'(\omega))(2\pi iq(\omega))^{n_i}.
$$
The lower order terms of $Q_i$ can be expressed in terms of $\chi_q$ and derivatives of the symbol of $\mathscr P_i$. However, due to the dependence of $\omega$ for example in $p_{in_i}(t/q'(\omega))$, the two factors on the right cannot be separated 
in such a way that
the operator $\FT\mathscr P_i\IFT$ 
is directly
realized as a finite difference operator. Investigating the case of time-dependent coefficients is therefore beyond the scope of the current paper and will not be pursued further here.

\subsubsection{Dynamics of wave packets}\label{ss:dynamics}
For $(x_0,\xi_0)\in T^\ast(\R)$, a function of the form
$$
\varphi_{(x_0,\xi_0)}(t)=(2/h)^{\frac{1}{4}}
e^{2\pi i (t-x_0)\xi_0/h}e^{-\pi (t-x_0)^2/h}
$$
will be called a {\it Gaussian wave packet}. Here 
$\varphi_{(x_0,\xi_0)}$ has been normalized with respect to the usual inner product in $L^2(\R)$. We see that when $h\ll 1$, $\varphi_{(x_0,\xi_0)}(t)=O(h^\infty)$ is negligible if $t\ne x_0$, where $O(h^\infty)$ means $O(h^N)$ for all $N>0$. Similarly,
\begin{equation}\label{eq:Fourierwavepacket}
\mathcal F_h(\varphi_{(x_0,\xi_0)} )(\eta)=
(2/h)^{\frac{1}{4}}e^{-2\pi i x_0\eta/h}e^{-\pi (\eta-\xi_0)^2/h}
\end{equation}
is negligible if $\eta\ne\xi_0$. These notions are combined in the following definition.

\begin{dfn}
Let $u=u(h)$, $0<h\le h_0$, be a family of functions in $L^2(\R)$. We say that $u$ is {\it microlocally small} near $(x_0,\xi_0)\in T^\ast(\R)$ if the inner product
$$
\lvert( u,\varphi_{(x,\xi)})_{L^2(\R)}\rvert=\bigg\lvert\int_\R u(t)e^{-2\pi i (t-x)\xi/h}e^{-\pi(t-x)^2/h} \,dt\bigg\rvert=O(h^\infty)
$$
uniformly for $(x,\xi)$ in a neighborhood of $(x_0,\xi_0)$. The complement of such points $(x_0,\xi_0)$ is called the {\it semiclassical wavefront set} of $u$, denoted $\mathrm{WF}_h(u)$.
\end{dfn}

Another common name for the semiclassical wavefront set is {\it frequency set}, usually denoted $\mathrm{FS}(u)$. For other equivalent definitions, including those employing the Fourier-Bros-Iagolnitzer (FBI) transform we refer to Martinez \cite{martinez2002introduction} and Zworski \cite{zworski2012semiclassical}. Our presentation is inspired by Faure \cite{faure2011semiclassical}.

As alluded to above, $\mathrm{WF}_h(\varphi_{(x_0,\xi_0)} )=\{(x_0,\xi_0)\}$ which is made evident by computing the inner product
$$
(\varphi_{(x_0,\xi_0)},\varphi_{(x,\xi)})_{L^2(\R)}=e^{-2\pi i (x_0-x)(\xi_0+\xi)/2h}e^{-\pi(x_0-x)^2/2h}e^{-\pi(\xi_0-\xi)^2/2h}.
$$
The following result describes how the wavefront set of a Gaussian wave packet is affected by the dispersion transforms.

\begin{proposition}\label{prop:wavefrontset}
Let $\varphi_{(x_0,\xi_0)}$ be a Gaussian wave packet, and let $\chi$ be the canonical map given by \eqref{eq:chi}. Then
\begin{align*}
\mathrm{WF}_h(\IFT(\varphi_{(x_0,\xi_0)}))&=\{\chi(x_0,\xi_0)\}=\{(x_0/\q'(\xi_0),\q(\xi_0))\},\\
\mathrm{WF}_h(\FT(\varphi_{(x_0,\xi_0)}))&=\{\chi^{-1}(x_0,\xi_0)\}=\{(x_0\q'(\q^{-1}(\xi_0)),\q^{-1}(\xi_0))\}.
\end{align*}
\end{proposition}

\begin{proof}
We will prove the first identity; the proof of the second is similar and is left to the reader. Changing variables in \eqref{semiclassicalI} we see that
\begin{equation*}
\IFT(\varphi_{(x_0,\xi_0)})=\mathcal F^{-1}_h(\mathbf 1_{\q(\On)} \mathcal F_h(\varphi_{(x_0,\xi_0)})(\q^{-1}(\cdot) )),
\end{equation*}
so an application of the Plancharel formula gives
\begin{align*}
( \IFT(\varphi_{(x_0,\xi_0)}),\varphi_{(x,\xi)})_{L^2(\R)}&=( \mathcal F_h(\IFT(\varphi_{(x_0,\xi_0)})),\mathcal F_h(\varphi_{(x,\xi)}))_{L^2(\R)}\\
&=( \mathbf 1_{\q(\On)} \mathcal F_h(\varphi_{(x_0,\xi_0)})(\q^{-1}(\cdot) ),\mathcal F_h(\varphi_{(x,\xi)}))_{L^2(\R)}.
\end{align*}
In view of \eqref{eq:Fourierwavepacket}, $( \IFT(\varphi_{(x_0,\xi_0)}),\varphi_{(x,\xi)})_{L^2(\R)}$ is therefore equal to the integral
\begin{equation}\label{eq:oscint}
\int_{\q(\On)} e^{-2\pi i x_0\q^{-1}(\eta)/h}e^{-\pi (\q^{-1}(\eta)-\xi_0)^2/h}e^{2\pi i x\eta/h}e^{-\pi (\eta-\xi)^2/h}\,d\eta.
\end{equation}
Due to the quadratic terms in the exponential, this is clearly $O(h^\infty)$ in the semiclassical limit $h\to0$ if there is no $\eta_0$ such that $\eta_0-\xi=\q^{-1}(\eta_0)-\xi_0=0$, i.e., such that $\xi=\eta_0=\q(\xi_0)$.
Writing the remaining oscillatory factors as $e^{2\pi i\phi(\eta)/h}$ with $\phi(\eta)=x\eta-x_0\q^{-1}(\eta)$, it follows from the principle of non-stationary phase  
that the integral is also $O(h^\infty)$ unless $\phi'(\eta_0)=0$, see e.g., \cite[Lemma 3.10]{zworski2012semiclassical}. But $\phi'(\eta_0)=0$ implies that
\begin{equation*}
x=\frac{x_0}{\q'(\q^{-1}(\eta_0))}=\frac{x_0}{\q'(\xi_0)},
\end{equation*}
so $\mathrm{WF}_h(\IFT(\varphi_{(x_0,\xi_0)}))\subset\{(x_0/\q'(\xi_0),\q(\xi_0))\}$. 
That we in fact have equality follows by applying the method of stationary phase to \eqref{eq:oscint} with $(x,\xi)$ determined above. The phase function 
$$
\Phi(\eta)=x_0\bigg(\frac{\eta}{\q'(\xi_0)}-\q^{-1}(\eta)\bigg)+\frac{i}{2}\bigg((\q^{-1}(\eta)-\xi_0)^2+(\eta-\q(\xi_0))^2\bigg)
$$
satisfies $\im\Phi\ge0$ and has a unique critical point at $\eta_0=\q(\xi_0)$, which is non-degenerate since
$$
\Phi''(\eta_0)=x_0\frac{\q''(\xi_0)}{(\q'(\xi_0))^3}+i\bigg(1+\frac{1}{(\q'(\xi_0))^2}\bigg)\ne0.
$$
Multiplying $e^{2\pi i\Phi/h}$ by a cutoff function which is identically $1$ near $\eta_0$ thus gives
$$
( \IFT(\varphi_{(x_0,\xi_0)}),\varphi_{(x,\xi)})_{L^2(\R)}=h^\frac12 e^{2\pi i\Phi(\eta_0)/h}(\Phi''(\eta_0)/i)^{-\frac12}+O(h^\frac32),
$$
see H{\"o}rmander \cite[Theorem 7.7.5]{hormander1983linear}. Since $\lvert e^{2\pi i\Phi(\eta_0)/h}\rvert=1$, the right-hand side is not smaller than $O(h^\frac12)$ as $h\to0$, which yields the result.
\end{proof}

\subsection{Initial conditions}\label{ss:init}

Let $u_i$ be the solution to \eqref{evo} with initial conditions \eqref{init}, so $u_i(t,x)\equiv0$ for $t\le0$. 
By virtue of Theorem \ref{thm:evo}, $\FT(u_i)$ satisfies the corresponding finite difference equation modified to account for time dispersion, namely \eqref{FDevo}. Suppose we want to use \eqref{FDevo} to obtain a sampling of $\FT(u_i)$. This would require adding initial conditions to \eqref{FDevo}, and it is natural to use the same initial conditions \eqref{init} as before. Let therefore $v_i$ be a function satisfying \eqref{FDevo} with initial conditions \eqref{init},
so that
$v_i(t,x)\equiv 0$ for $t\le0$. Then because the equation but not the initial condition for $v_i$ has been modified using the FTDT, this introduces an approximation error between $\FT(u_i)$ and $v_i$. Indeed, $v_i(t,x)\equiv0$ for $t\le 0$ while we have, by the Fourier inversion formula and the definition of $\FT(u_i)$,
\begin{equation}\label{psi_ij}
\FT(u_i)(t,x)
=\int_\Omega e^{2\pi it\omega}\widehat u_i(q(\omega),x)\, d\omega 
\end{equation}
which does not equal 
\begin{equation*}
u_i(t,x)
=\int_{-\infty}^\infty e^{2\pi i t\omega}\widehat u_i(\omega,x)\, d\omega,
\end{equation*}
and thus $\FT(u_i)(t,x)\ne u_i(t,x)\equiv 0$ for $t\le 0$.
We see that the ``correct'' finite difference initial value problem to solve in order to obtain a sampling of $\FT(u_i)$ would be \eqref{FDevo} with initial condition given by \eqref{psi_ij} for $t\le0$. Since this would involve using knowledge of $u_i$ this is obviously not possible in practice. 
On the other hand, since $\FT(u_i)$ and $v_i$ have the same evolution in time according to Theorem \ref{thm:evo}, $\FT(u_i)$ will continue to stay close to $v_i$ for all $t$ as long as $\FT(u_i)(t,x)$ is approximately identically zero for $t\le0$. The accuracy of approximation is ensured by the following lemma, expressed in terms of derivatives $\partial_t^j\FT(u_i)$ of orders $j$ for which the regularity assumption \eqref{regularityassumption} is valid.

To keep the presentation general, we make the assumptions that
\begin{itemize}
\item $\mathbf 1_\Omega(\omega)$ converges pointwise to 1 as $\Delta t\to0$, i.e., $\Omega\subset \R$ exhausts $\R$ in the limit as $\Delta t\to0$, 
\item $q(\omega)$ converges pointwise to $\omega$ as $\Delta t\to0$, 
\item $\lvert q(\omega)\rvert\ge c \lvert \omega\rvert$ for $\omega\in \Omega$ where $c$ is a real-valued constant independent of $\Delta t$.
\end{itemize}
We also assume that $\Omega$ has a decomposition $\Omega=\Omega_\mathrm{inn}\cup \Omega_\mathrm{out}$ consisting of an inner and outer region where  
$\Omega_\mathrm{inn}\to\R$ as $\Delta t\to0$, such that
for some real-valued constants $C_1,C_2,C_3$ independent of $\Delta t$,
\begin{itemize}
\item $q'(\omega)\ge C_1$ if $\omega\in \Omega_\mathrm{inn}$,
\item $\lvert \omega\rvert\ge C_2/\Delta t$ if $\omega\in \Omega_\mathrm{out}$,
\item $\Omega_\mathrm{out}$ has Lebesgue measure $\lvert \Omega_\mathrm{out}\rvert\le C_3/\Delta t$.
\end{itemize}
To illustrate, if $q(\omega)$ is the function described in Example \ref{ex:q} then these assumptions are satisfied with $c=2/\pi$, $\Omega_\mathrm{inn}=\{\omega:\lvert\omega\rvert\le (8\Delta t)^{-1}\}$ and $C_1=1/\sqrt 2$, $C_2=1/8$, $C_3=1/4$.

\begin{lemma}\label{lem:init}
Let $u_i$ solve \eqref{evo}--\eqref{init}. Then
\begin{equation*}
\lim_{\Delta t\to0}\sup_{t\le 0}{\lvert\partial_t^j\FT(u_i)(t,x)\rvert}=0,\quad 0\le j\le n_i-1,
\end{equation*}
and the convergence is uniform with respect to $x\in X$ in the sense of \eqref{regularityassumption}. The rate of convergence depends on the definition of the phase shift function $q$ in \eqref{def:q}.
\end{lemma}

\begin{proof}
Inspecting the definitions and recalling that $\Omega=\Omega_\mathrm{inn}\cup \Omega_\mathrm{out}$ we see in view of the Fourier inversion formula that the result is proved by showing
\begin{equation}\label{convinitcond}
\lim_{\Delta t\to0}\bigg(\int_{\Omega_\mathrm{inn}}+\int_{\Omega_\mathrm{out}}\bigg) e^{2\pi it\omega}(2\pi i\omega)^j\widehat u_i(q(\omega),x)\, d\omega=\infint e^{2\pi it\omega}(2\pi i\omega)^j\widehat u_i(\omega,x)\, d\omega
\end{equation}
with uniform convergence in $t$. Indeed, the left-hand side is the limit of $\partial_t^j\FT(u_i)(t,x)$ as $\Delta t\to0$ and the right-hand side is equal to $\partial_t^ju_i(t,x)$ which is zero for $t\le0$ by assumption. Note that \eqref{convinitcond} is essentially a consequence of the Lebesgue dominated convergence theorem. For the benefit of the reader we include the details.

Before treating each integral on the left separately we make two observations. First, \eqref{regularityassumption} implies
\begin{equation}\label{Sobineq}
\infint (1+\lvert\omega\rvert^2)^k\lvert\widehat u_i(\omega,x)\rvert^2\, d\omega<\infty,\quad 0\le k\le n_i,
\end{equation}
which means that $\lvert \widehat u_i(\omega,x)\rvert\le g_i(\omega,x)^{1/2}(1+\lvert \omega\rvert^2)^{-n_i/2}$ for some integrable function $\omega\mapsto g_i(\omega,x)$, where $g_i(\omega,x)\to0$ as $\lvert\omega\rvert\to\infty$ by the Riemann-Lebesgue lemma.
Second, by assumption we have $\lvert q(\omega)\rvert\ge c\lvert\omega\rvert$ for $\omega\in \Omega$ with $c$ independent of $\Delta t$, so
\begin{equation}\label{genest}
\int_{\Omega} \lvert(2\pi i\omega)^j\widehat u_i(q(\omega),x)\rvert\, d\omega\le
\bigg(\frac{2\pi}{c}\bigg)^j\int_{\Omega} \lvert  q(\omega)\rvert^j\lvert\widehat u_i(q(\omega),x)\rvert\, d\omega.
\end{equation}

We begin by treating the first integral on the left of \eqref{convinitcond}. Recall that by our standing assumptions, $u_i(t,x)$ is integrable in $t$ which means that $\widehat u_i(\omega,x)$ is continuous in $\omega$, while $\mathbf 1_{\Omega_\mathrm{inn}}(\omega)\to 1$ and $q(\omega)\to\omega$ pointwise as $\Delta t\to0$. Next, note that
$$
\int_{\Omega_\mathrm{inn}} \lvert  q(\omega)\rvert^j\lvert\widehat u_i(q(\omega),x)\rvert\, d\omega=\int_{q(\Omega_\mathrm{inn})}\frac{\lvert\xi\rvert^j\lvert\widehat u_i(\xi,x)\rvert}{q'(q^{-1}(\xi))}\,d\xi\le \frac{1}{C_1}\infint\lvert\xi\rvert^j\lvert\widehat u_i(\xi,x)\rvert\, d\xi
$$
since $\xi\in q(\Omega_\mathrm{inn})$ implies that $\omega=q^{-1}(\xi)\in \Omega_\mathrm{inn}$ for which we have $q'(\omega)\ge C_1$ by assumption. Since $C_1$ is independent of $\Delta t$ and the right-most integral is convergent by \eqref{Sobineq}, Lebesgue's dominated convergence theorem together with \eqref{genest} implies that
$$
\lim_{\Delta t\to0}\int_{\Omega_\mathrm{inn}} e^{2\pi it\omega}(2\pi i\omega)^j\widehat u_i(q(\omega),x)\, d\omega=\infint e^{2\pi it\omega}(2\pi i\omega)^j\widehat u_i(\omega,x)\, d\omega,
$$
uniformly in $t$.

To treat the second integral on the left of \eqref{convinitcond}, note that
$$
\lvert  q(\omega)\rvert^j\lvert\widehat u_i(q(\omega),x)\rvert
\le\frac{\lvert q(\omega)\rvert^j}{(1+\lvert q(\omega)\rvert^2)^{(n_i-1)/2}}\frac{g_i(\omega,x)^{1/2}}{(1+\lvert q(\omega)\rvert^2)^{1/2}}\le \frac{g_i(\omega,x)^{1/2}}{(1+c^2\omega^2)^{1/2}}
$$
for $0\le j\le n_i-1$. 
Since $\lvert \omega\rvert\ge C_1/\Delta t$ when $\omega\in \Omega_\mathrm{out}$ and $\lvert \Omega_\mathrm{out}\rvert\le C_2/\Delta t$, it is then easy to see that
$$
\int_{\Omega_\mathrm{out}} \lvert  q(\omega)\rvert^j\lvert\widehat u_i(q(\omega),x)\rvert\, d\omega\le \frac{C_2}{C_1c}\sup_{\omega\in \Omega_\mathrm{out}}g_i(\omega,x)^{1/2}
$$
with $g_i$ and $c$ independent of $\Delta t$.  Since $g_i(\omega,x)\to0$ as $\lvert\omega\rvert\to\infty$ it follows that the supremum above tends to 0 as $\Delta t\to0$. In view of \eqref{genest} we conclude that
$$
\lim_{\Delta t\to0}\int_{\Omega_\mathrm{out}}e^{2\pi it\omega}(2\pi i\omega)^j\widehat u_i(q(\omega),x)\, d\omega=0,
$$
uniformly in $t$.
This proves \eqref{convinitcond}.
From the proof it is clear that the convergence is uniform with respect to $x\in X$ in the sense of \eqref{regularityassumption}, and that the rate of convergence depends on $q(\omega)$. 
\end{proof}

\subsection{Non-matching finite difference schemes}\label{ss:nonmatching}

Here we briefly discuss what can happen if non-matching finite difference approximations are used to define $\DDj$ in \eqref{eq:defofFDsystem}. 
To highlight the effect we choose a simple prototype of the evolution equation \eqref{evo}: Let $u=(u_1,u_2)$ and consider the system
\begin{align}\label{maineq}
\partial_tu_1(t,x)&=L_{11}u_1(t,x)+L_{12}u_2(t,x),\\
\partial_tu_2(t,x)&=L_{21}u_1(t,x)-Cu_2(t,x),
\label{auxeq}
\end{align}
where the $L_{ij}$ are linear spatial operators independent of $t$ and $C$ is a constant. 
We will as before let $\DDe$ denote a scheme satisfying all prior assumptions and use it to model the time derivative in \eqref{maineq}, but we assume that \eqref{auxeq} is an auxiliary equation and allow a different scheme to model its time derivative. (An example is provided by the usage of memory variables in \eqref{sigma}--\eqref{memory}, where in large-scale supercomputer global seismological simulations it can be desirable that the stress and displacement are updated at every step in time, but the memory variables only every 4 steps in time, which saves computational costs without being dramatically worse in performance.)
Denote it by
$$
\DDt v_i(t,x)=\sum_n \widetilde c_{1,n} v_i(t+n\Delta t,x),
$$
and define 
$$
\widetilde q(\omega) = \frac{1}{2\pi i}\sum_n \widetilde c_{1,n}  e^{2\pi i \omega \Delta t n} .
$$
The next result explains how the discretized equations should be modified in order to accurately model the evolution of a solution to \eqref{maineq}--\eqref{auxeq}.

\begin{proposition}\label{prop:nonmatching}
Let $u=(u_1,u_2)$ be a solution of \eqref{maineq}--\eqref{auxeq}. 
Set
$$
G(t) = \int_\Omega \frac{2\pi i\widetilde q(\omega)+C}{2\pi iq(\omega)+C} e^{2\pi i t\omega } \, d\omega,
$$
and define $v_i=\FT(u_i)$ for $i=1,2$. Then $v=(v_1,v_2)$ solves 
\begin{align}\label{FDmaineq}
\DDe v_1(t,x)&=L_{11}v_1(t,x)+L_{12}v_2(t,x),\\
\DDt v_2(t,x)&=L_{21} v_1(\cdot,x)\ast G(t)-Cv_2(t,x),
\label{FDauxeq}
\end{align}
for each value of $t$, where $\ast$ denotes time convolution.
\end{proposition}
\begin{proof}
Note that $G$ is well defined since $q$ is real-valued and the integration domain is a compact set. Also,
\begin{equation}\label{Ghat}
\widehat G(\omega)=\mathbf 1_\Omega(\omega)\frac{2\pi i\widetilde q(\omega)+C}{2\pi iq(\omega)+C}.
\end{equation}
Fix $x$ and suppress it from the notation, and let $L_j$ be the $1\times 2$ system $L_j=(L_{j1},L_{j2})$, $j=1,2$. Using the Fourier inversion formula and the definition of $v_1$ we have
\begin{align*}
\DDe v_1(t)-L_1v(t) &=
\int_\Omega  \left[(2\pi i q(\omega) -L_{11})\widehat{u}_1(q(\omega))-L_{12}\widehat{u}_2(q(\omega))\right]  e^{2\pi i t\omega} \, d \omega.
\end{align*}
Taking a Fourier transform of \eqref{maineq} and evaluating at $q(\omega)$ shows that
$$
\left[(2\pi iq(\omega)-L_{11})\widehat u_1(q(\omega))-L_{12}\widehat{u}_2(q(\omega))\right]=0,
$$
so $v_1$ solves \eqref{FDmaineq}.

To prove that $v_2$ solves \eqref{FDauxeq}, we observe that
\begin{equation}\label{observedformula}
\widehat u_2(q(\omega))=\frac{1}{2\pi iq(\omega)+C}L_{21}\widehat u_1(q(\omega)).
\end{equation}
This formula is easily obtained by taking a Fourier transform of \eqref{auxeq}, solving for $\widehat u_2$ and evaluating the result at $q(\omega)$. It follows that
$$
\DDt v_2(t)+Cv_2(t)=\infint (2\pi i \widetilde q(\omega)+C)\widehat v_2(\omega)e^{2\pi i t\omega } \, d\omega.
$$
Since $\widehat v_2(\omega)=\mathbf 1_\Omega(\omega)\widehat u_2(q(\omega))$ we find in view of \eqref{Ghat} and \eqref{observedformula} that
$$
\DDt v_2(t)+Cv_2(t)=\infint \widehat G(\omega)L_{21}\widehat u_1(q(\omega))\mathbf 1_\Omega(\omega)e^{2\pi i t\omega } \, d\omega.
$$
Since $\widehat u_1(q(\omega))\mathbf 1_\Omega(\omega)=\widehat v_1(\omega)$, this is equivalent to \eqref{FDauxeq} by virtue of the Fourier inversion formula.
\end{proof}

Proposition \ref{prop:nonmatching} shows that the price one has to pay for using different finite difference schemes to approximate the time derivatives in \eqref{maineq} and in \eqref{auxeq}, is the appearance of a convolution in \eqref{FDauxeq}. Ignoring the convolution results in an approximation of the desired evolution that can be estimated in terms of the amount by which \eqref{Ghat} differs from 1.

Note that if the constant $C$ in \eqref{FDauxeq} is replaced by a spatial operator $L_{22}$ which, while independent of $t$, is not simply constant in $x$ then the previous result has to be modified accordingly. By minor changes, the proof of Proposition \ref{prop:nonmatching} shows that the result remains valid if $G$ is replaced by
$$
G(t) = \int_\Omega \frac{\widetilde q(\omega)}{q(\omega)} e^{2\pi i t\omega } \, d\omega,
$$
and \eqref{FDauxeq} is replaced by
$$
\DDt v_2(t,x)=(L_{21} v_1(\cdot,x)+L_{22} v_2(\cdot,x))\ast G(t).
$$
We remark that $G$ is well defined due to the assumption that $\lvert q(\omega)\rvert\ge c\lvert\omega\rvert$ for $\omega\in\Omega$.

\section{Viscoelastic finite difference equations}\label{3dvisco}

Here we discuss the removal of time dispersion from 2D and 3D viscoelastic finite difference modeling for a specific leapfrog scheme developed in \cite{robertsson1994viscoelastic} (see Bohlen \cite{bohlen2002parallel} for an explicit implementation).
Recall \eqref{Newton2nd3D}--\eqref{memory}.
The time derivative of a function is approximated by
\begin{equation}\label{eq:BohlenD}
\DDe f(t,x)=\frac{f(t+\frac{1}{2}\Delta t)-f(t-\frac{1}{2}\Delta t)}{\Delta t}.
\end{equation}
In this case, the phase shift function $q(\omega)$ is found to be
\begin{equation}\label{qBohlen}
q(\omega)=\frac{ \sin(\pi\omega\Delta t)}{\pi\Delta t},
\end{equation}
which is invertible for 
$\omega\in\Omega$ where $\Omega=[-\frac{1}{2\Delta t},\frac{1}{2\Delta t}]$, see Example \ref{ex:2ndorder}. Here the upper limit $1/2\Delta t$ coincides with the Nyquist frequency which is an improvement compared to finite difference scheme employed in 
Section \ref{section:numsim}. The drawback is the need to use a time average of the memory variables as described below.

Let $\M f(t,x)$ denote the time average
\begin{equation}\label{average}
\M f(t,x)=\tfrac{1}{2}(f(t+\tfrac{1}{2}\Delta t,x)+f(t-\tfrac{1}{2}\Delta t,x))
\end{equation}
of a function $f(t,x)$.
Equations \eqref{Newton2nd3D}--\eqref{memory} are discretized in time by
\begin{equation}\label{app:FDNewton3D}
\rho\DDe V_i=\partial_j\Sigma_{ij}+g_i,
\end{equation}
together with
\begin{align}\label{app:FDsigma}
\DDe\Sigma_{ij}&=\partial_kV_k\{\pi(1+\tau^p)-2\mu(1+\tau^s)\}+2\partial_jV_i\mu(1+\tau^p)\\&\quad+\frac{1}{N}\sum_{n=1}^N\M R_{ij n}\quad \text{if }i=j,\notag\\
\DDe\Sigma_{ij}&=\big(\partial_jV_i+\partial_iV_j\big)\mu(1+\tau^s)+\frac{1}{N}\sum_{n=1}^N\M R_{ij n}\quad \text{if }i\ne j,
\notag
\end{align}
and
\begin{align}\label{app:FDmemory}
\DDe R_{ijn}&=-\frac{1}{\tau_{\sigma n}}\left\{ \M R_{ijn}+(\pi\tau^p-2\mu\tau^s)\partial_kV_k+2\partial_jV_i\mu\tau^s\right\}\quad\text{if }i=j,\\
\DDe R_{ijn}&=-\frac{1}{\tau_{\sigma n}}\left\{ \M R_{ijn}+\mu\tau^s(\partial_jV_i+\partial_iV_j)\right\}\quad\text{if }i\ne j.\notag
\end{align}
If in addition the spatial derivatives are discretized using a fourth-order staggered forward operator and backward operator as is done by Bohlen \cite{bohlen2002parallel}, one arrives at the discrete equations \cite[(A.3)--(A.17)]{bohlen2002parallel} after some straightforward calculations. (In this paper we have extended it to a twelfth-order scheme.) We then have the following.

\begin{theorem}\label{thm:memory}
Let $v_i$ and $\sigma_{ij}$ solve \eqref{Newton2nd3D}--\eqref{sigma}, with memory variables solving \eqref{memory}. 
Define $V_i=\FT(v_i)$, $\Sigma_{ij}=\FT(\sigma_{ij})$ and $g_i=\FT(f_i)$.
Then for each value of $t$, $V_i$ and $\Sigma_{ij}$ solve \eqref{app:FDNewton3D} exactly and \eqref{app:FDsigma} approximately, where the $R_{ijn}$ are exact solutions to \eqref{app:FDmemory}. In the numerical simulations of Section \ref{section:viscoelastic}, the approximation error is $O(\Delta t^2)$. 
\end{theorem}

Before the proof we recall from \S\ref{ss:memory} that, according to Theorem \ref{thm:evo}, the functions $V_i$, $\Sigma_{ij}$ and $R_{ijn}$ are exact solutions to the equations obtained by removing all occurrences of the averaging operator $\M$ from \eqref{app:FDNewton3D}--\eqref{app:FDmemory}.

\begin{proof}
We will keep $x$ fixed and omit it from the notation. We prove the statements when $i\ne j$, the other case being similar. We first observe that for a function $f(t)$ we have
\begin{equation}\label{Fourieraverage}
\widehat{\M f}(\omega)=\widehat f(\omega)\cos(\pi\omega\Delta t).
\end{equation}
We also record the fact that if $v_i$ and $\sigma_{ij}$ solve \eqref{Newton2nd3D}--\eqref{memory} then 
\begin{gather}\label{FourierNewton2nd3D}
\rho(2\pi i\omega)\widehat{v}_i(\omega)=\partial_j\widehat\sigma_{ij}(\omega)+\widehat f_i(\omega),
\\
\label{Fouriersigma}
(2\pi i\omega)\widehat{\sigma}_{ij}(\omega)=\big(\partial_j\widehat v_i(\omega)+\partial_i\widehat v_j(\omega)\big)\mu(1+\tau^s)+\frac{1}{N}\sum_{n=1}^N\widehat r_{ij n}(\omega),
\end{gather}
where
\begin{equation}\label{Fouriermemory}
\widehat{r}_{ijn}(\omega)=-\frac{\mu\tau^s(\partial_j\widehat{v}_i(\omega)+\partial_i\widehat v_j(\omega))}{1+2\pi i\omega\tau_{\sigma n}},\quad 1\le n\le N,
\end{equation}
which follows from \eqref{memory} and a straightforward computation.

By definition, $\widehat{V}_i(\omega)=\mathbf 1_\Omega (\omega) \widehat{v}_i(q(\omega))$. Similar formulas hold for $\Sigma_{ij}$ and $g_i$. Hence,
$$
\rho\,\DDe V_i(t)-\partial_j\Sigma_{ij}(t) = \int_\Omega \big[\rho(2\pi i q(\omega)) \widehat{v}_i(q(\omega))-\partial_j\widehat{\sigma}_{ij}(q(\omega))\big]  e^{2\pi i t\omega} \, d \omega
$$
by the Fourier inversion formula.
Using \eqref{FourierNewton2nd3D} evaluated at $q(\omega)$ instead of $\omega$ we see that the right-hand side is equal to $\int_\Omega \widehat{f}_i(q(\omega)) e^{2\pi i t\omega} \, d \omega=g_i(t)$, which proves that \eqref{app:FDNewton3D} holds.

Next, write
\begin{multline*}
\DDe \Sigma_{ij}(t)-\big(\partial_jV_i(t)+\partial_iV_j(t)\big)\mu(1+\tau^s)\\=\int_\Omega \big[2\pi i q(\omega) \widehat{\sigma}_{ij}(q(\omega))-\big(\partial_j\widehat v_i(q(\omega))+\partial_i\widehat v_j(q(\omega))\big)\mu(1+\tau^s)\big]  e^{2\pi i t\omega} \, d \omega.
\end{multline*}
Applying \eqref{Fouriersigma} evaluated at $q(\omega)$ instead of $\omega$ we get
\begin{equation}\label{proofi=/j}
\DDe \Sigma_{ij}(t)-\big(\partial_jV_i(t)+\partial_iV_j(t)\big)\mu(1+\tau^s)
=\int_\Omega\Bigg( \frac{1}{N}\sum_{n=1}^N \widehat{r}_{ijn}(q(\omega)) \Bigg)  e^{2\pi i t\omega} \, d \omega. 
\end{equation}
We now take a Fourier transform in $t$ of \eqref{app:FDmemory}. Using \eqref{Fourieraverage}, elementary computations show that
$$
\widehat{R}_{ijn}(\omega)=-\frac{\mu\tau^s(\partial_j\widehat V_i(\omega)+\partial_i\widehat V_j(\omega))}{\cos(\pi\omega\Delta t)+2\pi iq(\omega)\tau_{\sigma n}}
=\widehat{r}_{ijn}(q(\omega))\frac{1+2\pi iq(\omega)\tau_{\sigma n}}{\cos(\pi\omega\Delta t)+2\pi iq(\omega)\tau_{\sigma n}}
$$
for $\omega\in\Omega$, where the last identity follows by inserting the definition of $V_i$ and inspecting \eqref{Fouriermemory}.
Using \eqref{Fourieraverage} again it is straightforward to check that
$$
\widehat{r}_{ijn}(q(\omega))=\widehat{\M R}_{ijn}(\omega)+\widehat{R}_{ijn}(\omega)(1-\cos(\pi\omega\Delta t))\frac{2 i\sin(\pi\omega\Delta t)\tau_{\sigma n}}{\Delta t+2 i\sin(\pi\omega\Delta t)\tau_{\sigma n}}.
$$
where the last factor is uniformly bounded for $\omega\in\Omega$, and the second factor is $O(\Delta t^2)$ when $\omega$ is restricted to a bounded, $\Delta t$-independent set.
In the simulations in Section \ref{section:viscoelastic} it turns out that $\widehat R_{ijn}(\omega)$ is indeed supported in a $\Delta t$-independent set, see Figure \ref{figure:mv1}. In view of \eqref{proofi=/j} we thus conclude that 
$$
\DDe \Sigma_{ij}(t)-\big(\partial_jV_i(t)+\partial_iV_j(t)\big)\mu(1+\tau^s)
=\int_\Omega\bigg( \frac{1}{N}\sum_{\ell=1}^N \widehat{\M R}_{ijn}(\omega) \bigg)  e^{2\pi i t\omega} \, d \omega+O(\Delta t^2).
$$
The result now follows by applying the Fourier inversion formula to the integral on the right.
\end{proof}

Naturally, there are also versions of Theorems \ref{thm:inverse} and \ref{thm:mainthm} corresponding to Theorem \ref{thm:memory}, as well as a version of the converse statement in Theorem \ref{thm:evo}. We leave for the reader to fill in the details.

\begin{figure}
\centering
\includegraphics{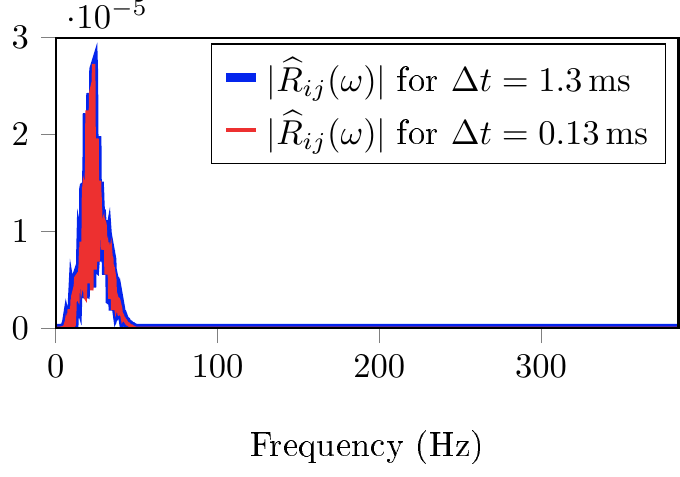}
\caption{The spectrum of the summed, dispersed memory variables $R_{ij}=\sum_n R_{ijn}$ as recorded at $(z,x)=(262.5,62.5)$ in the viscoelastic simulations from Section \ref{section:viscoelastic}. Note how the energy fits well within the Nyquist frequency $1/(2\cdot0.0013)=385\,\mathrm{Hz}$, even for the coarsest possible time-step $\Delta t=1.3\,\mathrm{ms}$. Moreover, the energy remains limited to about $40\,\mathrm{Hz}$, also for simulations with a smaller $\Delta t=0.13\,\mathrm{ms}$. \label{figure:mv1}}
\end{figure}

\section{Implementation}\label{app:matlab}
Here we show how to implement the discrete dispersion transforms in MATLAB in two specific cases, namely the finite difference scheme from Example \ref{ex:implementation2} that is used in the numerical simulations of Section \ref{section:numsim},  and the leapfrog scheme from Appendix \ref{3dvisco} that is used in the viscoelastic simulations of Section \ref{section:viscoelastic}. The interested reader should then be able to adapt the procedure to other cases without difficulty.

\subsection{Central difference scheme}\label{ss:centralFD}
Consider the finite difference operator from Example \ref{ex:implementation2} and
recall \eqref{discreteITDTfast}. We see by inspection that we can view $(\ITDT (f_n))_{k}$ in matrix terms as row $k+1$ of a matrix $A$ applied to 
$\tilde f=(f_0,\ldots,f_{N-1},0,\ldots,0)$,
where $A=(a_{(k+1)(n+1)})_{k,n=0}^{2N-1}$ is the matrix with element
$$
a_{(k+1)(n+1)}=\frac{1}{2N}\sum_{m=-N/2}^{N/2}e^{-2\pi i nm/2N}e^{ i k\sin(\pi m/N) }\cos(\pi m/N)
$$
at position $(k+1)(n+1)$. We may view the sum as ranging over $-N\le m\le N-1$ with terms for $-N/2+1\le \lvert m\rvert\le N$ being zero; in particular the term for $m=-N$ is zero, and so would the term with $m=N$ be. Changing variables $m\mapsto -m$ we thus see that this is the inverse discrete Fourier transform of $m\mapsto g_k(m)$ evaluated at $n$, where
$$
g_k(m)=\begin{cases}e^{-i k\sin(\pi m/N)}\cos(\pi m/N)&\text{ for $\lvert m\rvert \le N/2$},\\ 0&\text{ otherwise.}\end{cases}
$$
$A$ can therefore be computed by applying the inverse FFT to the matrix with columns  $g_{k+1}$ and taking real transpose, e.g., via
\begin{lstlisting}
function [A] = finite_difference_ITDT(N)
A1 = exp( -1i * sin([0:N/2]*pi/(N))' * [0:2*N-1] );
A  = ifft(  cos([0:N/2]'*pi/N) .* A1, 2*N, 'symmetric' )';
A  = A(1:N,1:N);
end
\end{lstlisting}
where we also take advantage of conjugate symmetry.
The last line truncates the matrix so it can be applied directly to the original vector $f$ without having to zeropad the sample manually as this is already built in.

Similarly, by inspecting \eqref{discreteFTDTfast} we see that we can view $(\FTDT (f_n))_{k}$ as row $k+1$ of a matrix $B$ applied to $\tilde f=(f_0,\ldots,f_{N-1},0,\ldots,0)$, where $B=(b_{(k+1)(n+1)})_{k,n=0}^{2N-1}$ is the matrix with element
$$
b_{(k+1)(n+1)}=\frac{1}{2N}\sum_{m=-N/2}^{N/2}e^{ -i n\sin(\pi m/N) }e^{2\pi i mk/2N}
$$
at position $(k+1)(n+1)$. As before we may view the sum as ranging over $-N\le m\le N-1$ with terms for $-N/2+1\le \lvert m\rvert\le N$ being zero. We thus see that this is the inverse discrete Fourier transform of $m\mapsto h_n(m)$ evaluated at $k$, where 
$$
h_n(m)=\begin{cases}e^{-i n\sin(\pi m/N)}&\text{ for $\lvert m\rvert \le N/2$},\\ 0&\text{ otherwise.}\end{cases}
$$
$B$ can therefore be computed (without taking transpose) by applying the inverse FFT to the matrix with columns $h_{n+1}$, e.g., via
\begin{lstlisting}
function [B] = finite_difference_FTDT(N)
B1 = exp( -1i * sin([0:N/2]*pi/N)' * [0:2*N-1] );
B = ifft( B1, 2*N, 'symmetric');
B = B(1:N,1:N);
end
\end{lstlisting}
As before, the last line truncates the matrix thus avoiding the need to zeropad the sample $f$ manually. The matrices $A$ and $B$ are depicted in Figure \ref{fig:matlab}.

\begin{figure}
\centering
\includegraphics{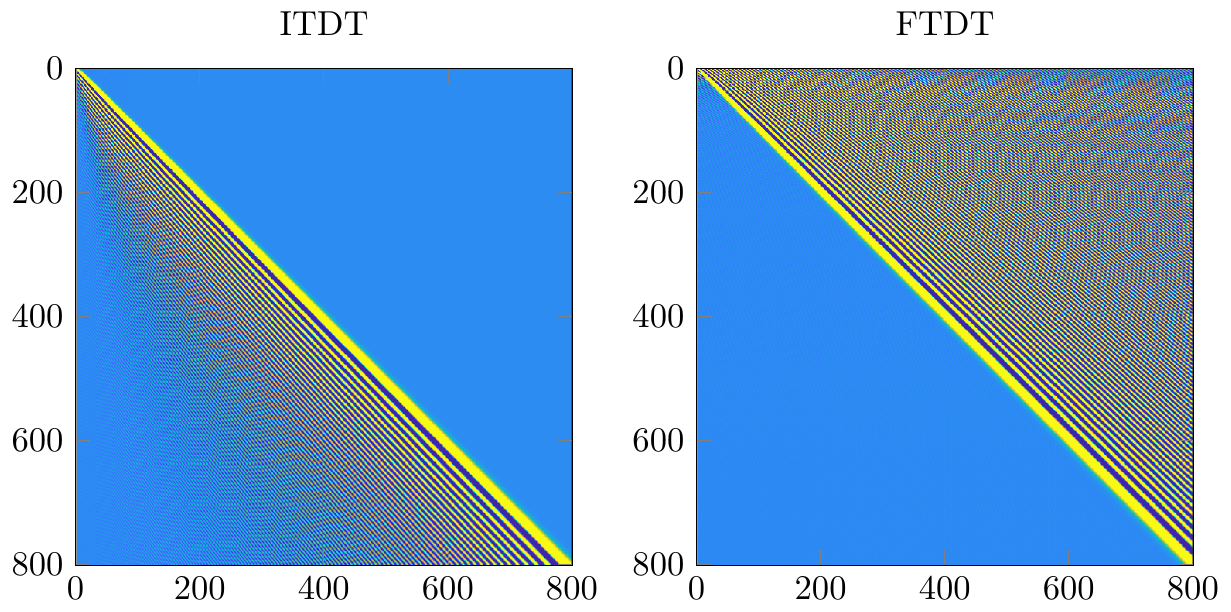}
\caption{Visual illustration of matrix representations of the inverse time dispersion transform (left) and the forward time dispersion transform (right). \label{fig:matlab}}
\end{figure}

\subsection{Leapfrog scheme}\label{ss:leapfrog}
Consider now the leapfrog scheme from Appendix \ref{3dvisco} and recall that in this case, the phase shift function is $q(\omega)=\sin(\pi\omega\Delta t)/\pi\Delta t$ by \eqref{qBohlen}, so $q$ is invertible for $\omega\in\Omega$ where $\Omega=[-\frac{1}{2\Delta t},\frac{1}{2\Delta t}]$. We remark that this is not the same as replacing $\Delta t$ by $\Delta t/2$ in the previous subsection since the time-step size for each individual function is in fact $\Delta t$. Repeating the arguments in Example \ref{ex:implementation2} for this choice of $q$ we find by inserting
$$
\omega_m=\frac{1}{2\Delta t}\frac{m}{N},\quad m=-N+1,\ldots,N,
$$
into \eqref{discreteITDT} that
\begin{equation*}
(\ITDT (f_n))_k=\frac{1}{2N}\sum_{m=-N+1}^{N}\Bigg(\sum_{n=0}^{N-1}e^{-2\pi i nm/2N}f_n\Bigg)e^{2i k\sin(\pi m/2N) }\cos(\pi m/2N).
\end{equation*}
Here the inner sum is the value at $m$ of the discrete Fourier transform of the vector $(f_0,\ldots,f_{N-1})$ zeropadded to twice the length. The outer sum is the value at $k$ of a modified discrete inverse Fourier transform (without truncation). This is the image of $f$ under the action of row $k+1$ of the matrix $A=(a_{(k+1)(n+1)})_{k,n=0}^{2N-1}$ with element
$$
a_{(k+1)(n+1)}=\frac{1}{2N}\sum_{m=-N+1}^{N}e^{-2\pi i nm/2N}e^{ 2i k\sin(\pi m/2N) }\cos(\pi m/2N)
$$
at position $(k+1)(n+1)$. Changing variables $m\mapsto -m$ we identify this as the inverse discrete Fourier transform of $m\mapsto g_k(m)$ evaluated at $n$, where
$$
g_k(m)=e^{-2i k\sin(\pi m/2N)}\cos(\pi m/2N)\quad\text{ for $m=-N,\ldots, N-1$}.
$$
$A$ can be computed by applying the inverse FFT to the matrix with columns  $g_{k+1}$ and taking real transpose via
\begin{lstlisting}
function [A] = leapfrog_ITDT(N)
A1 = exp( -2i * sin([0:N-1]*pi/(2*N))' * [0:2*N-1] );
A  = ifft(  cos([0:N-1]'*pi/(2*N)) .* A1, 2*N, 'symmetric' )';
A  = A(1:N,1:N);
end
\end{lstlisting}

Next, by inserting the expression for $\omega_m$ (shifted one index) into \eqref{discreteFTDT} we find that
\begin{equation*}
(\FTDT (f_n))_k=\frac{1}{2N}\sum_{m=-N}^{N-1}\Bigg(\sum_{n=0}^{N-1}e^{-2i n\sin(\pi m/2N) }f_n\Bigg)e^{2\pi i mk/2N }.
\end{equation*}
Thus we can view $(\FTDT (f_n))_{k}$ as row $k+1$ of a matrix $B$ applied to $(f_0,\ldots,f_{N-1})$ zeropadded to twice the length, where $B=(b_{(k+1)(n+1)})_{k,n=0}^{2N-1}$ has element
$$
b_{(k+1)(n+1)}=\frac{1}{2N}\sum_{m=-N}^{N-1}e^{ -2i n\sin(\pi m/2N) }e^{2\pi i mk/2N}
$$
at position $(k+1)(n+1)$. This is the inverse discrete Fourier transform of $m\mapsto h_n(m)$ evaluated at $k$, where 
$$
h_n(m)=e^{-2i n\sin(\pi m/2N)}\quad\text{ for $m=-N,\ldots, N-1$}.
$$
$B$ can be computed by applying the inverse FFT to the matrix with columns $h_{n+1}$ via
\begin{lstlisting}
function [B] = leapfrog_FTDT(N)
B1 = exp( -2i * sin([0:N-1]*pi/(2*N))' * [0:2*N-1] );
B = ifft( B1, 2*N, 'symmetric');
B = B(1:N,1:N);
end
\end{lstlisting}

\bibliographystyle{amsplain}
\bibliography{referenser}

\providecommand{\bysame}{\leavevmode\hbox to3em{\hrulefill}\thinspace}
\providecommand{\MR}{\relax\ifhmode\unskip\space\fi MR }
\providecommand{\MRhref}[2]{%
  \href{http://www.ams.org/mathscinet-getitem?mr=#1}{#2}
}
\providecommand{\href}[2]{#2}
\begin{thebibliography}{10}

\bibitem{amundsen2019elimination}
Lasse Amundsen and {\O}rjan Pedersen, \emph{Elimination of temporal dispersion
  from the finite-difference solutions of wave equations in elastic and
  anelastic models}, Geophysics \textbf{84} (2019), no.~2, T47--T58.

\bibitem{blanch1995modeling}
Joakim~O. Blanch, Johan O.~A. Robertsson, and William~W. Symes, \emph{Modeling
  of a constant q: Methodology and algorithm for an efficient and optimally
  inexpensive viscoelastic technique}, Geophysics \textbf{60} (1995), no.~1,
  176--184.

\bibitem{bohlen2002parallel}
Thomas Bohlen, \emph{Parallel 3-d viscoelastic finite difference seismic
  modelling}, Computers \& Geosciences \textbf{28} (2002), no.~8, 887--899.

\bibitem{SOFI2Dmanual}
Thomas Bohlen, Denise De~Nil, Daniel K\"{o}hn, and Stefan Jetschny,
  \emph{Sofi2d seismic modeling with finite differences: 2d -- elastic and
  viscoelastic version}, Karlsruhe Institute of Technology, 2016.

\bibitem{evans2010partial}
Lawrence~C. Evans, \emph{Partial differential equations}, second ed., Graduate
  texts in Mathematics, vol.~19, American Mathematical Society, 2010.

\bibitem{faure2011semiclassical}
Fr{\'e}d{\'e}ric Faure, \emph{Semiclassical origin of the spectral gap for
  transfer operators of a partially expanding map}, Nonlinearity \textbf{24}
  (2011), no.~5, 1473--1498.

\bibitem{fornberg1998practical}
Bengt Fornberg, \emph{A practical guide to pseudospectral methods}, vol.~1,
  Cambridge university press, 1998.

\bibitem{gao2016third}
Yingjie Gao, Jinhai Zhang, and Zhenxing Yao, \emph{Third-order symplectic
  integration method with inverse time dispersion transform for long-term
  simulation}, Journal of Computational Physics \textbf{314} (2016), 436--449.

\bibitem{hairer2006geometric}
Ernst Hairer, Christian Lubich, and Gerhard Wanner, \emph{Geometric numerical
  integration: structure-preserving algorithms for ordinary differential
  equations}, vol.~31, Springer Science \& Business Media, 2006.

\bibitem{hoffmann2000computational}
Klaus~A. Hoffmann and Steve~T. Chiang, \emph{Computational fluid dynamics
  volume i}, Engineering Education System, 2000.

\bibitem{holberg1987computational}
Olav Holberg, \emph{Computational aspects of the choice of operator and
  sampling interval for numerical differentiation in large-scale simulation of
  wave phenomena}, Geophysical prospecting \textbf{35} (1987), no.~6, 629--655.

\bibitem{hormander1983linear}
Lars H{\"o}rmander, \emph{{The analysis of linear partial differential
  equations I--IV}}, Springer Verlag, 1983--1985.

\bibitem{kindelan1990construction}
M.~Kindelan, A.~Kamel, and P.~Sguazzero, \emph{On the construction and
  efficiency of staggered numerical differentiators for the wave equation},
  Geophysics \textbf{55} (1990), no.~1, 107--110.

\bibitem{koene2017removing}
Erik F.~M. Koene and Johan O.~A. Robertsson, \emph{Removing numerical
  dispersion artifacts from reverse time migration and full-waveform
  inversion}, SEG Technical Program Expanded Abstracts 2017, Society of
  Exploration Geophysicists, 2017, pp.~4143--4147.

\bibitem{koene2018eliminating}
Erik F.~M. Koene, Johan O.~A. Robertsson, Filippo Broggini, and Fredrik
  Andersson, \emph{Eliminating time dispersion from seismic wave modeling},
  Geophysical Journal International \textbf{213} (2018), no.~1, 169--180.

\bibitem{koene2019eliminating}
Erik F.~M. Koene, Jens Wittsten, Johan O.~A. Robertsson, and Fredrik Andersson,
  \emph{Eliminating time dispersion from visco-elastic simulations with memory
  variables}, 81st EAGE Conference and Exhibition 2019, 2019.

\bibitem{lax1956survey}
Peter~D. Lax and Robert~D. Richtmyer, \emph{Survey of the stability of linear
  finite difference equations}, Communications on pure and applied mathematics
  \textbf{9} (1956), no.~2, 267--293.

\bibitem{leveque2007finite}
Randall~J. LeVeque, \emph{Finite difference methods for ordinary and partial
  differential equations: steady-state and time-dependent problems}, vol.~98,
  Siam, 2007.

\bibitem{martin2002marmousi}
Gary~S. Martin, Kurt~J. Marfurt, and Shawn Larsen, \emph{Marmousi-2: An updated
  model for the investigation of avo in structurally complex areas}, SEG
  Technical Program Expanded Abstracts 2002, Society of Exploration
  Geophysicists, 2002, pp.~1979--1982.

\bibitem{martinez2002introduction}
Andr{\'e} Martinez, \emph{An introduction to semiclassical and microlocal
  analysis}, Springer, 2002.

\bibitem{mittet2019second}
Rune Mittet, \emph{Second-order time integration of the wave equation with
  dispersion correction procedures}, Geophysics \textbf{84} (2019), no.~4,
  T221--T235.

\bibitem{moczo2007}
Peter Moczo, Johan O.~A. Robertsson, and Leo Eisner, \emph{The
  finite-difference time-domain method for modeling of seismic wave
  propagation}, Advances in Geophysics \textbf{48} (2007), 421--516.

\bibitem{anderson1984}
Richard~H. Pletcher, John~C. Tannehill, and Dale~A. Anderson,
  \emph{Computational fluid mechanics and heat transfer}, McGraw-Hill, 1984.

\bibitem{robertsson1994viscoelastic}
Johan O~A Robertsson, Joakim~O Blanch, and William~W Symes, \emph{Viscoelastic
  finite-difference modeling}, Geophysics \textbf{59} (1994), no.~9,
  1444--1456.

\bibitem{stork2013eliminating}
Christof Stork, \emph{Eliminating nearly all dispersion error from fd modeling
  and rtm with minimal cost increase}, 75th EAGE Conference \& Exhibition
  incorporating SPE EUROPEC 2013, 2013.

\bibitem{titarev2002ader}
Vladimir~A. Titarev and Eleuterio~F. Toro, \emph{Ader: Arbitrary high order
  godunov approach}, Journal of Scientific Computing \textbf{17} (2002),
  no.~1-4, 609--618.

\bibitem{virieux1986p}
Jean Virieux, \emph{P-sv wave propagation in heterogeneous media:
  Velocity-stress finite-difference method}, Geophysics \textbf{51} (1986),
  no.~4, 889--901.

\bibitem{wang2015finite}
Meixia Wang and Sheng Xu, \emph{Finite-difference time dispersion transforms
  for wave propagation}, Geophysics \textbf{80} (2015), no.~6, WD19--WD25.

\bibitem{zworski2012semiclassical}
Maciej Zworski, \emph{Semiclassical analysis}, vol. 138, American Mathematical
  Soc., 2012.

\end{thebibliography}

\end{document}